\documentclass[default]{sn-jnl}
\usepackage{amsmath, amsthm, amscd, amsfonts, amssymb, graphicx, color, subcaption, longtable, slashbox,pict2e,float}
\usepackage{caption}
\usepackage{mathtools}
\usepackage{hyperref}
\usepackage{MnSymbol}

\jyear{2021}%

\theoremstyle{thmstyleone}%
\newtheorem{theorem}{Theorem}[section]
\newtheorem{proposition}[theorem]{Proposition}%

\theoremstyle{thmstyletwo}%

\theoremstyle{thmstylethree}%
\newtheorem{definition}{Definition}[section]%
\newtheorem{example}{Example}[section]%
\newtheorem{remark}{Remark}[section]%
\newtheorem{lemma}{Lemma}[section]
\newtheorem{corollary}{Corollary}[section]
\numberwithin{equation}{section}
\newtheorem{solution}{Solution}[section]%
\renewenvironment{proof}{{\bfseries Proof}}{\qed}
\renewenvironment{solution}{{\bfseries solution}}{\qed}
\begin{document}

\title[Log-ergodicity: A New Concept for Modeling Financial Markets]{Log-ergodicity: A New Concept for Modeling Financial Markets}


\author*[1]{\fnm{Kiarash} \sur{Firouzi}}\email{kiarashfirouzi91@gmail.com}

\author[2]{ \fnm{Mohammad} \spfx{Jelodari} \sur{Mamaghani}}\email{j\_mamaghani@atu.ac.ir}


\affil*[1]{\orgdiv{Department of Mathematics}, \orgname{Allameh Tabataba'i Unversity}, \orgaddress{\street{Dehkadeh Olympic}, \city{Tehran}, \postcode{1489684511}, \state{Tehran}, \country{Iran}}}

\affil[2]{\orgdiv{Department of Mathematics}, \orgname{Allameh Tabataba'i Unversity}, \orgaddress{\street{Dehkadeh Olympic}, \city{Tehran}, \postcode{1489684511}, \state{Tehran}, \country{Iran}}}



\abstract{
Although financial models violate ergodicity in general, observing the ergodic behavior in the markets is not rare. Policymakers and market participants control the market behavior in critical and emergency states, which leads to some degree of ergodicity as their actions are intentional. In this paper, we define a parametric operator that acts on the space of positive stochastic processes, transforming a class of positive stochastic processes into mean-ergodic processes. With this mechanism, we extract the data regarding the ergodic behavior hidden in the financial model, apply it to mathematical finance, and establish a novel method for pricing contingent claims. We provide some empirical examples and compare the results with existing ones to demonstrate the efficacy of this new approach.


}

\keywords{Ergodic maker operator, Log-ergodic process, Mean-ergodic, Partially ergodic, Time-average}


\pacs[MSC Classification]{37A30, 37H05, 60G10, 91B70, 91G15}

\maketitle
\section{Introduction}\label{sec1}
During the financial crises, we have experienced that governments and policymakers control the market instabilities. The ``Wall Street bailout", which reduced the effects of the financial market crisis of 2007-2008 \cite{bail} is an example. 

One might ask about the reflection of these actions in financial market mathematical models. In other words, what are the corresponding concept(s) of these interventions in financial market mathematical models?

In fact, from the mathematical point of view, they do nothing but direct the models to be mean-reverting, bounded, less volatile, and so on. Therefore, for the financial models to be usable in such situations, some parts of them have to be deleted using some appropriate mathematical tools. In this regard, we point out the paper \cite{dc} in which Fuqi Chen and colleagues have conducted a comprehensive analysis of the controls on financial markets regarding the drift coefficient, which indicates the timewise inhibition of risky assets as changing the rate of the drift coefficient affects the duration of market cycles. 

To participate in controlling the market model irregularities, in this paper, we introduce the new concepts of the log-ergodic process and the ergodic maker operator. The one-parameter, ergodic maker operator produces a mean-ergodic process when it acts on a positive stochastic process. This operator reflects the controls regarding the volatility of risky assets.




The notion of mean is one of the common concepts between mathematical finance and ergodic theory. In the first, it enters as an expectation in most price computations, and in the second, it plays the fundamental role of defining the Birkhoff notion of ergodicity.

Before we proceed further, let us mention that for a model (process) to be ergodic, it must have Markov property with a stationary distribution. Additionally, the model must possess the mean recurrence property to be ergodic \cite{39,40,35}. By definition, a stochastic process is ergodic in the mean, or simply mean-ergodic, if its ensemble-average and time-average are equal in the long run \cite{1,59}.

Ole Peters has presented a thorough analysis of ergodic economics \cite{89}. Additionally, since 2011, the London Mathematical Laboratory has conducted specialized research on ergodic economics \cite{90}. Some research has been on modeling blockchain-enabled economics using stochastic dynamical systems \cite{100}. 

Looking at financial stochastic processes from an ergodic theory point of view one may ask: Which financial market models are ergodic? Which non-ergodic financial models can be made into an ergodic model? Which non-ergodic processes have the potential to turn into an ergodic process? 

Some random processes with specific properties are ergodic or at least mean-ergodic. Markov processes with stationary distributions are ergodic \cite{59}. Oesook Lee demonstrated an example of the mixing and ergodic properties for generalized Ornstein-Uhlenbeck processes \cite{93}. Paper \cite{as} applies the assumption of ergodicity to obtain specific estimates for asymptotic arbitrage, demonstrating their connection to large deviation estimates for the market price of risk. It further explores the geometric Ornstein-Uhlenbeck process as an example. Trabelsi explored the ergodic properties of the $CIR$ model and demonstrated that it has the ergodic recurrence property\footnote{An irreducible, non-periodic Markov chain with a stationary distribution is said to be recurrent if it converges to its stationary distribution for almost all initial points.} (which is also known as the mean reversion property) \cite{62}. The $CIR$ process is ergodic and has a stationary distribution \cite{62}. Hiroshi Kunita discusses various aspects of stochastic flows and their relation to ergodic theory and stochastic differential equations in \cite{kuni}. 


In this paper, we study some algebraic properties of the ergodic maker operator and show that it preserves some algebraic operations on stochastic processes. We also provide examples of log-ergodic stochastic processes that are helpful in modeling mean-ergodic financial markets. Also, we discuss the applications of log-ergodic processes to price contingent claims. To this end, we derive a partial differential equation under the usual assumptions regarding the price function that depends on the ergodic maker operator. 
Furthermore, we study the effects of market restrictions on the price dynamics and volatility of risky assets using log-ergodic processes.

The rest of the paper is organized as follows:\\ 
In section \ref{sec2}, we review some necessary concepts from ergodic theory, ergodic economics, and stochastic calculus. 
 In section \ref{sec4}, we define the concept of the ergodic maker operator and the log-ergodic process and investigate their properties. In section \ref{sec5}, we present examples of log-ergodic processes that can be used to model financial markets with ergodic behavior in the mean. In section \ref{sec6}, we state and prove the main theorem. In section \ref{sec7}, we discuss the applications of log-ergodic processes in pricing contingent claims and studying market restrictions in this respect. In section \ref{sec8}, we present the empirical data analysis of our study. In section \ref{sec9}, we conclude the paper and suggest some directions for future research.
\section{Preliminaries}\label{sec2}
From now on, we use the filtered probability space
$(\Omega,\mathcal{F},\mathbb{P},(\mathcal{F}_t)_{t\geq0})$,  in which   $\Omega$ is the space of events, $\mathcal{F}$ is a $\sigma$-algebra,  $\mathbb{P}$ is an invariant probability measure\footnote{If a stochastic process has an invariant measure, then the distribution of the process at any time will be the same as the distribution of the process at any other time.} (see \cite{1,poll} for definition), and $(\mathcal{F})_{{t\geq0}}$ is a filtration which represents the information of the financial market up to time $t$. 

As is well known, there are two requirements for a homogeneous Markov process $X_t$ to be ergodic. First, its time and ensemble averages should be equal. Second, time and ensemble averages of its autocorrelation function should be the same \cite{106}. The process is referred to as mean-ergodic if just the first criterion holds.

\begin{theorem}(Birkhoff)\label{Bir}
If $\mathbb{P}$ is a probability measure invariant under a stochastic process $X_t$ and $\phi\in L^1(\mathbb{P})$, then the function 
	\begin{equation}\label{1}
		\tilde{\phi}(\omega)=\lim_{T\rightarrow \infty}\frac{1}{T}\int_0^T\phi(X_t(\omega))dt
	\end{equation}
is defined almost surely and $\int_\Omega\tilde{\phi}d\mathbb{P}=\int_{\Omega}\phi d\mathbb{P}$.
\end{theorem}
\begin{proof}
For the proof and more details refer to \cite{114}.
\end{proof}


Considering $\phi$ as the identity function $\phi=I$ in theorem \ref{Bir}, yields
$\phi(X_t(\omega))=X_t(\omega)$. As a result, we have
\begin{align}\label{tav}
	\tilde{I}(\omega)=\lim_{T\rightarrow \infty}\frac{1}{T}\int_0^T X_t(\omega) dt.
\end{align}
We call $\tilde{I}$ the time-average of the process $X_t$ and denote it by $<X>$.

Suppose that $X_t$ is a Hölder continuous positive stochastic process of order $0<\beta<\infty$. i.e.
\begin{equation*}
	\exists b>0;\quad \lvert X_t-X_s\rvert\leq b\lvert t-s\rvert^{\beta},\quad \forall t,s>0. 
\end{equation*}
According to the exponential decay of correlation theorem \cite{1}, there exist positive numbers $\Lambda<1$, and $\alpha(X_s,X_t)$ such that the correlation coefficient $\mathbf{cor}(X_s,X_t)$ satisfies the following relationship.
\begin{align*}
	\mathbf{cor}(X_s,X_t)&\leq \alpha(X_s,X_t)\Lambda^n,\quad \forall n\geq 1.
\end{align*}
Therefore, from the definition of the correlation, we have
\begin{align}
	\frac{\mathbf{Cov}(X_s,X_t)}{\Lambda^n\sqrt{\mathbb{V}ar[X_t]\mathbb{V}ar[X_s]}}&\leq \alpha(X_s,X_t),\label{alpha}
\end{align}
where $\mathbf{Cov}(X_s,X_t)$ is the covariance of $X_t$ and $X_s$.

\subsection{Ergodicity and Utility Functions in Economics}\label{sec3}
Let $X_t$ represent the wealth process of an investor. The primary problem of ergodic economics is to analyze the evolution of this process. Ergodic economics assumes that investor choices will optimize the time-average of the growth rate of the $X_t$ process. From Daniel Bernoulli's conjecture \cite{105}, it follows that the utility of each additional dollar is almost inversely proportional to the number (units) of dollars that the investor currently has \cite{89}. Therefore, the growth rate of $X_t$ is governed by the differential equation $dU(X_t)=\frac{1}{X_t}d{X_t}$ with initial condition $U(0)=\ln(X_0)$, and the solution $U(X_t)=\ln(X_t)$, in which $X_0\neq 0$ \cite{89}. Under these circumstances, let $g$ be the growth rate of $X_t$ and write $g_{t,X_t}=\frac{\Delta U(X_t)}{\Delta t}$. 

Although processes of type $X_t$ generally violate the ergodic property, their growth rates are ergodic \cite{89}. We observe that the time-average of the growth rate of $X_t$ is defined using the mathematical expectation of the variation of $U(X_t)$, which leads us to the following definition:
\begin{definition}
	The time-average of the growth rate of a stochastic wealth process $X_t$ is defined as:
	$$<g_{t,X_t}>=\frac{\mathbb{E}[\Delta U(X_t)]}{\Delta t},$$
	where $U(X_t)=\ln(X_t)$.
\end{definition}
A simple model of the wealth process of the investor, $X_t$, widely used in mathematical finance and other fields is the geometric Brownian motion \cite{13}.
\begin{example}
	Consider the process
	\begin{equation*}
		X_t=X_0\exp\{(\mu-\frac{1}{2}\sigma^2)t+\sigma W_t\},\quad X_0=x\neq 0,
	\end{equation*}
	where the constants $\mu$ and $\sigma$ are the drift and volatility coefficients of the process, respectively, $x$ is a real number, and $W_t$ is a standard Wiener process. We observe that
	\begin{align}
		U(X_t)&=\ln(X_t),\label{3}\\
		\ln(X_t)&=\ln(X_0)+(\mu-\frac{\sigma^2}{2})t+\sigma W_t.\label{2}
	\end{align} 
	 It follows that $U(X_t)$ has a linear growth concerning time and its time-average of the growth rate is
	\begin{equation*}
		<g_{t,X_t}>=\frac{\mathbb{E}[\Delta U(X_t)]}{\Delta t}=\mu-\frac{1}{2}\sigma^2.
	\end{equation*}
	Therefore, maximizing the rate of change of the logarithmic utility function \ref{3} is equivalent to maximizing the time-average of the growth rate of the wealth \cite{89}.

	Due to the necessity of using the function $U(\cdot)$, from now on, we will study the logarithm of the positive processes used in financial theory.
\end{example}
\subsection{Market Cycles and Volatility Control}
Market cycles are price and economic activity fluctuations that happen over time in response to market factors such as supply and demand, interest rates, innovations, sentiment, and shocks \cite{cycles}. They consist of phases like expansion, peak, and contraction and can differ in duration, intensity, and frequency \cite{83}.

Market cycles affect the volatility of risky assets in several ways. During periods of expansion, when the economy is growing and the market is optimistic, the volatility of risky assets tends to be low, as the prices tend to move in a steady upward direction. The demand for risky assets increases as investors seek higher returns and are willing to take more risk. The supply of risky assets may also increase as innovation and productivity create new opportunities and products. During peak periods, when the economy is at its highest level of output with the market being euphoric, the volatility of risky assets may begin to rise as prices become overvalued and unsustainable. The demand for risky assets may exceed supply, leading to bubbles and speculation. The supply of risky assets may also decrease as innovation and productivity slow down or face constraints. During periods of contraction, when the economy is shrinking with the market being pessimistic, the volatility of risky assets tends to be high, as the prices fall sharply and unpredictably. The demand for risky assets decreases as investors seek lower returns and are unwilling to take more risk. The supply of risky assets may also increase as innovation and productivity create new challenges and risks.




Although the Brownian motion process is not of bounded variation, it is one of the processes that the market participants and policymakers control its variations in financial markets \cite{103}. Issuing currencies, supplying and removing liquidity from the markets, and adopting stringent legislation are a few of these controls \cite{103,104}. Diffusion models are the most popular models of financial markets. In this paper, considering the diffusion models, we show that a suitable ergodic maker operator measures the degree of control exerted by these factors in the markets.
\section{Ergodic Maker Operator and the Log-Ergodic Processes}\label{sec4}
In this section, we introduce the concepts of the ergodic maker operator (EMO) and the log-ergodic process and investigate some of their properties.

As in \cite{roy}, for the stochastic process $X_t$, we denote the mean-square convergence by $ms\lim_{t\rightarrow\infty}X_t=X$, and the convergence in probability by $st\lim_{t\rightarrow \infty}X_t=X$.
\begin{proposition}\label{ID}
	Suppose that for the random process $X_t$, we have\\ $\lim_{t\rightarrow \infty}\mathbb{E}[X_t]=k$ and $\lim_{t\rightarrow \infty}\mathbb{V}ar[X_t]=0$. Then,
	\begin{equation*}
		ms\lim_{t\rightarrow\infty}X_t=k.
	\end{equation*}
\end{proposition}
\begin{proof}
	See \cite{76}.
\end{proof}

The following theorem shows how the Wiener process (the Brownian motion) transforms into an ergodic process by adjusting its fluctuations according to a parameter $\beta$. We use this parameter to reflect the level of influence that market participants have on the price dynamics of a risky asset.
\begin{theorem}\label{IDV}
	For $\beta>\frac{1}{2}$ we have
	\begin{equation}
		ms\lim_{T\rightarrow\infty}\frac{W_t}{t^\beta}=0.
	\end{equation}
\end{theorem}
\begin{proof}
	For every $t>0$ we have
	$\frac{\mathbb{E}[W_t]}{t^{\beta}}=0$, and
	$$\frac{1}{t^{2\beta}}\mathbb{V}ar[W_t]=\frac{t}{t^{2\beta}}=\frac{1}{t^{2\beta-1}}.$$
	It follows that $ms\lim_{t\rightarrow\infty}\frac{W_t}{t^{\beta}}=0$.
	For more details refer to \cite{76}.
\end{proof}
\begin{corollary}\label{IDV2}
	If $M_t=\int_0^tW_sds$, and $\beta>\frac{3}{2}$, then
	\begin{equation*}
		ms\lim_{t\rightarrow\infty}\frac{M_t}{t^{\beta}}=0.
	\end{equation*}
\end{corollary}
\begin{proof}
	For $t>0$ we have $\frac{\mathbb{E}[M_t]}{t^{\beta}}=0$, and 	$$\frac{1}{t^{2\beta}}\mathbb{V}ar[M_t]=\frac{t^3}{3t^{2\beta}}=\frac{1}{3t^{2\beta-3}}.$$
	Now the result follows from proposition \ref{ID}.
\end{proof}
\begin{corollary}
	For $\beta>\frac{1}{2}$ we have: $st\lim_{t\rightarrow \infty}\frac{W_t}{t^{\beta}}=0.$
\end{corollary}
\begin{proof}
	See \cite{76}.
\end{proof}

Accordingly, the coefficient $1/{t^{\beta}}$, with $\beta>\frac{3}{2}$, inhibits (controls) the variations of the Wiener process. 

\subsection{The Ergodic Maker Operator}
Let $Y_t$ be a one-dimensional Itô process given by
$$Y_t=Y_0+\int_0^t\sigma_s dW_s+\int_0^t \mu_s ds, \quad Y_0=0.$$
Where $W_t$ is a standard Wiener process, and $\mu_t$ and $\sigma_t$ are drift and volatility coefficients, respectively. These coefficients are integrable functions of $t$ and $y$. Using the definition of the one-dimensional Itô process (page 44 of \cite{21}) we have
\begin{equation}\label{ito}
	\int_0^t (\sigma_s^2+\lvert\mu_s\rvert)ds<\infty.
\end{equation}
Define the positive stochastic process $X_t$ by
\begin{equation}\label{X}
	X_t=xe^{Y_t},\quad X_0=x.
\end{equation} 
Let
\begin{equation*}
	Y_t^\prime:=\ln(X_t)=\ln(x)+Y_t.
\end{equation*}
Then, we have
\begin{align}\label{itop}
	Y_t^\prime&=Y_0^\prime+\int_0^t\sigma_s dW_s+\int_0^t \mu_s ds;\quad Y_0^\prime=\ln(x)+Y_0.
\end{align}
The process $Y_t^\prime$ is not necessarily ergodic (since it is not always stationary \cite{13}); to define the Ergodic Maker Operator (EMO) we write $Y_t^\prime$ as the sum:
\begin{align*}
	Y_t^\prime&=\underbrace{Y_0^\prime}_{\text{constant}}+\underbrace{\int_0^t\mu_s ds}_{D_t}+\underbrace{\int_0^t\sigma_s dW_s}_{R_t},\\
	Y_t^\prime&=Y_0^\prime+D_t+R_t.
\end{align*}
\begin{definition}(EMO)\label{deferc}
	Let $W_t$ be a standard Wiener process and $\beta>\frac{3}{2}$. For all $t,s\in[0,T]$, we define the ergodic maker operator of the process $Y_t^\prime$ as
	\begin{equation*}
		\xi_{t-s,W_{t-s}}^{\beta}[Y_t^\prime]:=0\cdot Y_0^\prime+\frac{W_T}{T^{\beta}}\cdot D_{t-s}+\frac{1}{T^{\beta}}\cdot R_{t-s},
	\end{equation*}
	From now on, for all $t>s$, we denote the length of the time interval $[s,t]$ by $\delta=t-s$. Therefore, we have 
	\begin{equation}\label{xi}
		\xi_{\delta,W_\delta}^{\beta}[Y_t^\prime]:=0\cdot Y_0^\prime+\frac{W_T}{T^{\beta}}\cdot D_\delta+\frac{1}{T^{\beta}}\cdot R_\delta.
	\end{equation}
\end{definition}
Note that in the rest of the paper, we denote the process constructed using the EMO by $Z_\delta$.
\begin{remark}
	Since $\delta$ is the length of the time interval $[s,t]$, the process\\ $Z_\delta=\xi_{\delta,W_\delta}^{\beta}[Y_t^\prime]$ can be interpreted as a scale of the variation of the logarithm of the price of a risky asset concerning the parameter $\beta$.  
\end{remark}
\begin{definition}
	We define the inhibition degree $\beta$ as 
	\begin{equation}\label{beta}
		\beta:=\begin{cases}
			\alpha, &\text{if} \quad \alpha>\frac{3}{2},\\
			\frac{3}{2}+\lvert\alpha\rvert, &\text{if} \quad \alpha<\frac{3}{2},
		\end{cases}
	\end{equation}
	where $\alpha$ satisfies \ref{alpha}.
\end{definition}

\subsubsection{Some Properties of the EMO}

Because the sample functions from an ergodic process are statistically equivalent, an ergodic process is stationary \cite{114}.

In the following lemma, we prove that the process made by the EMO is wide-sense stationary. This property is a direct consequence of mean-ergodicity \cite{114}.
\begin{lemma}\label{mm}
	$\xi_{\delta,W_\delta}^\beta[Y_t^\prime]$ is a wide-sense stationary stochastic process.
\end{lemma}	
\begin{proof}
	Let $Z_\delta=\xi_{\delta,W_\delta}^\beta[Y_t^\prime]$ and $\delta=t-s$ for $t>s$. 
We have
	\begin{equation}\label{zt}
		Z_\delta=Z_0+\frac{1}{T^{\beta}}\int_0^\delta\sigma_s dW_s+\frac{W_T}{T^{\beta}}\int_0^\delta\mu_sds,\quad Z_0=0.
	\end{equation}
	Calculating the expectation of the process $Z_\delta$ for $\delta$ and $\delta+\tau$ yields
	\begin{align*}
		\mathbb{E}[Z_\delta]=\mathbb{E}[Z_{\delta+\tau}]=0,\quad \forall \delta,\tau>0.  
	\end{align*}
Furthermore,
\begin{align*}
	\mathbb{E}[Z_\delta^2]&=\frac{1}{T^{2\beta}}\mathbb{E}\Big[\big(\int_0^\delta \sigma_sdW_s\big)^2+(W_T)^2\big(\int_0^\delta\mu_sds\big)^2\notag\\
	&+2W_T\big(\int_0^\delta \sigma_sdW_s\big)\big(\int_0^\delta \mu_sds\big)\Big]\\
	&=\frac{1}{T^{2\beta}}\mathbb{E}\Big[\big(\int_0^\delta \sigma_sdW_s\big)^2\Big]+\frac{1}{T^{2\beta}}\mathbb{E}\Big[(W_T)^2\big(\int_0^\delta\mu_sds\big)^2\Big]\\
	\text{Itô isometry}\Rightarrow &=\frac{1}{T^{2\beta}}\mathbb{E}\Big[\int_0^\delta \sigma_s^2ds\Big]+\frac{1}{T^{2\beta-1}}\mathbb{E}\Big[\big(\int_0^\delta\mu_sds\big)^2\Big].
\end{align*} 
From \ref{ito} we observe that
$$\mathbb{E}\Big[\int_0^\delta \sigma_s^2ds\Big]<\infty, \quad \text{and}\quad \mathbb{E}\Big[\big(\int_0^\delta\mu_sds\big)^2\Big]<\infty.$$
Therefore, $\mathbb{E}[Z_\delta^2]<\infty$.\\
Since the process $Z_\delta$ depends on $\delta=t-s$ (not on $t$ and $s$ individually), the correlation function of $Z_\delta$ is also a function of $\delta$ for all $t,s>0$. Therefore, the stochastic process $Z_\delta=\xi_{\delta,W_\delta}^\beta[Y_t^\prime]$ is wide-sense stationary.
\end{proof}
\begin{proposition}
	Suppose $H_t$ and $G_t$ are positive stochastic processes and let\\ $Y_t=\ln(H_t)$ and $Z_t=\ln(G_t)$. Then, the following statements hold.
		\item [1.] For all $a\in \mathbb{R}$ we have
		\begin{equation}\label{xi1}
			\xi_{\delta,W_\delta}^\beta[aY_t]=a\xi_{\delta,W_\delta}^\beta[Y_t].
		\end{equation}
		\item [2.] 
		\begin{equation}\label{xi2}
			\xi_{\delta,W_\delta}^\beta[Y_t+Z_t]=\xi_{\delta,W_\delta}^\beta[Y_t]+\xi_{\delta,W_\delta}^\beta[Z_t].
		\end{equation}
		\item [3.] 
		$\xi_{\delta,W_\delta}^\beta[Y_t\cdot Z_t]=D_\delta^z \xi_{\delta,W_\delta}^\beta[Y_t]+R_\delta^z L(\delta,W_\delta)$,\\
		 where $L(\delta,W_\delta)$ is a stochastic process to be found in the course of the proof.
\end{proposition}
\begin{proof}
	\begin{itemize}
		\item [1.] Using the definition \ref{xi} yields
		\begin{align*}
			\xi_{\delta,W_\delta}^\beta[aY_t]&= \frac{W_T}{T^{\beta}}\big(a\int_0^\delta \mu_s ds\big)+\frac{1}{T^\beta}\big(a\int_0^\delta \sigma_sdW_s\big)\\
			&=a\Big(\frac{W_T}{T^{\beta}}\big(\int_0^\delta \mu_s ds\big)+\frac{1}{T^\beta}\big(\int_0^\delta \sigma_sdW_s\big)\Big)\\
			&=a\big(\frac{W_T}{T^\beta}D_\delta+\frac{1}{T^\beta}R_\delta\big)=a\xi_{\delta,W_\delta}^\beta[Y_t].
		\end{align*}
		\item[2.] Suppose $D_t^y$ and $D_t^z$ are the deterministic parts of the processes $Y_t$ and $Z_t$, and $R_t^y$ and $R_t^z$ are the random parts of the processes $Y_t$ and $Z_t$, respectively. Using the EMO yields 
		\begin{align*}
			\xi_{\delta,W_\delta}^\beta[Y_t+Z_t]&= \frac{W_T}{T^\beta}(D_\delta^y+D_\delta^z)+\frac{1}{T^\beta}(R_\delta^y+R_\delta^z)\\
			&=\frac{W_T}{T^\beta}D_\delta^y+\frac{1}{T^\beta}R_\delta^y+\frac{W_T}{T^\beta}D_\delta^z+\frac{1}{T^\beta}R_\delta^z\\
			&=\xi_{\delta,W_\delta}^\beta[Y_t]+\xi_{\delta,W_\delta}^\beta[Z_t].
		\end{align*}
		\item [3.] Using the proof of \ref{xi2} we have
		\begin{align*}
			\xi_{\delta,W_\delta}^\beta[Y_tZ_t]&=\xi_{\delta,W_\delta}^\beta[D_t^yD_t^z+D_t^yR_t^z+R_t^yD_t^z+R_t^yR_t^z],
		\end{align*}
		which by \ref{xi1} can be written as
		\begin{align*}
			\xi_{\delta,W_\delta}^\beta[Y_tZ_t]&=\xi_{\delta,W_\delta}^\beta[D_t^yD_t^z]+\xi_{\delta,W_\delta}^\beta[D_t^yR_t^z]+\xi_{\delta,W_\delta}^\beta[R_t^yD_t^z]+\xi_{\delta,W_\delta}^\beta[R_t^yR_t^z]\\
			&=\frac{W_T}{T^\beta}(D_\delta^yD_\delta^z+R_\delta^yR_\delta^z)+\frac{1}{T^\beta}(D_\delta^yR_\delta^z+R_\delta^yD_\delta^z)\\
			&=D_\delta^z\big(\frac{W_T}{T^\beta}D_\delta^y+\frac{1}{T^\beta}R_\delta^y\big)+R_\delta^z\big(\frac{W_T}{T^\beta}R_\delta^y+\frac{1}{T^\beta}D_\delta^y\big)\\
			&=D_\delta^z \xi_{\delta,W_\delta}^\beta[Y_t]+R_\delta^z L(\delta,W_\delta).
		\end{align*}
		Where $L(\delta,W_\delta)=\frac{W_T}{T^\beta}R_\delta^y+\frac{1}{T^\beta}D_\delta^y$.
	\end{itemize}
\end{proof}
\subsection{Log-Ergodic Processes}
\begin{definition}(Log-ergodic process)\label{logergodic}
	The positive stochastic process $X_t$ is log-ergodic, if its log process, $Y_t=\ln(X_t)$, satisfies
	\begin{equation}\label{limer}
		\overline{<Y>}:=\lim_{T\rightarrow \infty}\frac{1}{T}\int_0^T (1-\frac{\tau}{T})\mathbf{Cov}_{yy}(\tau)d\tau=0, \quad \forall\tau\in[0,T].
	\end{equation}
	Where $\mathbf{Cov}_{yy}(\tau)$ is the covariance of  $Y_\tau$.
\end{definition}
\begin{definition}(Partial ergodicity)
	The positive stochastic process $X_t$ is partially ergodic if 
	$\xi_{\delta,W_\delta}^\beta[Y_t]$ satisfies \ref{limer}.
	\end{definition}
\begin{proposition}
	The linear combination of two log-ergodic processes is log-ergodic. 
\end{proposition}
\begin{proof}
		Consider the independent positive stochastic processes $H_t$ and $G_t$, and suppose that $Y_\delta=\xi_{\delta,W_\delta}^\beta[\ln(H_t)]$, and $Z_\delta=\xi_{\delta,W_\delta}^\beta[\ln(G_t)]$.  Then, for all real numbers $\gamma$ and $\nu$, it suffices to prove
	\begin{equation}
		\overline{<\gamma Y+\nu Z>}=\gamma^2\overline{<Y>}+\nu^2\overline{<Z>}.\label{I}
	\end{equation}
		We take $\delta_T=T-0=T$. Then, for every small time interval of length $\delta$ 
		we have
		\begin{align*}
			&\overline{<\gamma Y+\nu Z>}=\lim_{T\rightarrow\infty}\frac{1}{T}\int_0^T(1-\frac{\delta}{T})\Big(\mathbb{E}[\gamma^2Y_\delta^2+\nu^2Z_\delta^2+2\gamma\nu Y_\delta Z_\delta]-\gamma^2\mathbb{E}[Y_\delta]^2-\nu^2\mathbb{E}[Z_\delta]^2\Big)d\delta\\
			=&\lim_{T\rightarrow\infty}\frac{1}{T}\int_0^T(1-\frac{\delta}{T})(\mathbb{E}[\gamma^2Y_\delta^2]+\mathbb{E}[\nu^2Z_\delta^2])d\delta\\
			=&\lim_{T\rightarrow\infty}\frac{1}{T}\int_0^T(1-\frac{\delta}{T})\gamma^2\mathbb{E}[Y_\delta^2]d\delta+\lim_{T\rightarrow\infty}\frac{1}{T}\int_0^T(1-\frac{\delta}{T})\nu^2\mathbb{E}[Z_\delta^2]d\delta\\
			=&\lim_{T\rightarrow\infty}\frac{\gamma^2}{T}\int_0^T(1-\frac{\delta}{T})\mathbf{Cov}_{yy}(\delta)d\delta+\lim_{T\rightarrow\infty}\frac{\nu^2}{T}\int_0^T(1-\frac{\delta}{T})\mathbf{Cov}_{zz}(\delta)d\delta\\
			=&\gamma^2\overline{<Y>}+\nu^2\overline{<Z>}.
		\end{align*}
\end{proof}
\begin{proposition}
	Suppose that $H_t$ and $G_t$ are two independent positive log-ergodic processes with
	$\mathbb{E}[H_t]=m<\infty$ and $\mathbb{E}[G_t]= n<\infty$. Let $Y_t=\ln(H_t)$ and $Z_t=\ln(G_t)$. Then,
		\item[1.] $G_t+H_t$ is mean ergodic.
		\item[2.] $\nu H_t$ is log-ergodic for any real number $\nu$.
		\item[3.] 	$G_t\cdot H_t$ is log-ergodic.
\end{proposition}
\begin{proof}
		\item[1.] Define $K_t=G_t+H_t$. Calculating the covariance of $K_t$ we have:
		\begin{align*}
			\mathbf{Cov}_{kk}&=\mathbb{E}[K_t^2]-(\mathbb{E}[K_t])^2\\
			&=\mathbb{E}[G_t^2+H_t^2+2G_tH_t]-(\mathbb{E}[G_t+H_t])^2\\
			&=\mathbb{E}[G_t^2]+\mathbb{E}[H_t^2]+2\mathbb{E}[G_t]\mathbb{E}[H_t]-(\mathbb{E}[G_t]+\mathbb{E}[H_t])^2\\
			&=\mathbb{E}[G_t^2]-(\mathbb{E}[G_t])^2+\mathbb{E}[H_t^2]-(\mathbb{E}[H_t])^2\\
			&=\mathbb{E}[G_t-\mathbb{E}[G_t]]^2+\mathbb{E}[H_t-\mathbb{E}[H_t]]^2\\
			&=\mathbf{Cov}_{gg}+\mathbf{Cov}_{hh}=0+0=0.
		\end{align*} 
		Substituting the calculated covariance in \ref{limer}, we obtain $\overline{<K>}=0$.
		\item[2.] Using \ref{I} and \ref{limer} we have:
		\begin{align*}
			\overline{<\ln(\nu H)>}=&\overline{<\ln(\nu)+Y>}\\
			=&\overline{<\ln(\nu)>}+\overline{<Y>}=0+\overline{<Y>}=0.
		\end{align*}
		\item[3.] Using \ref{I} and \ref{limer} yields
		\begin{align*}
			\overline{<\ln(G\cdot H)>}=&\overline{<\ln(G)+\ln(H)>}\\
			=&\overline{<\ln(G)>}+\overline{<\ln(H)>}\\
			=&\overline{<Y>}+\overline{<Z>}=0.
		\end{align*}
\end{proof}
\begin{theorem}\label{thr}
	Let $X_t$ be a positive stochastic log-ergodic process. Then, there exist time intervals of length $\delta_i$, with $i\geq0$, for the process $Z_\delta=\xi_{\delta,W_\delta}^{\beta}[\ln(X_t)]$, in which the process is recurrent to its mean along any arbitrary path.
\end{theorem}
\begin{proof}
	We prove the theorem concerning a fixed path $\omega_0$. Using the definition of log-ergodicity, it follows that the relation \ref{limer} holds for the process $Z_\delta$. Therefore, from the definition of mean ergodicity\cite{1,59}, we have
	\begin{equation*}
		\lim_{T\rightarrow\infty}\frac{1}{T}\int_0^T Z_\delta d\delta=\mathbb{E}[Z_\delta].
	\end{equation*}
	It follows from Poincaré recurrence theorem \cite{1} that the process $Z_\delta$ is recurrent to its mean along the path. Therefore, there exists at least a time interval of length $\delta_0$ in $[0,T]$ such that $Z_{\delta_0}(\omega_0)=\mathbb{E}[Z_{\delta_0}(\omega_0)]$, $\mathbb{P}$-almost surely. For $i\in \mathbb{N}\cup\{0\}$, let $\{\delta_i\}_{i\geq0}$ represent the length of the time intervals in which the process $Z_\delta$ meets its mean along the path $\omega_0$ (as shown in Figure \ref{fig3}). It can be written that:
	\begin{equation*}
		\mathbb{P}\big(\exists i \in \mathbb{N}\cup\{0\}, \quad Z_{\delta_i}(\omega_0)=\mathbb{E}[Z_{\delta_i}(\omega_0)]\big)=1.
	\end{equation*}
	Consequently, using Birkhoff's ergodic theorem, as $T$ approaches infinity, it follows that there exist infinitely many time intervals of length $\delta_i$ for every $i\geq0$, for which $Z_\delta$ returns to its mean along the path $\omega_0$.
\end{proof}
\begin{figure}[H]
	\begin{center}
		\includegraphics[scale=0.5]{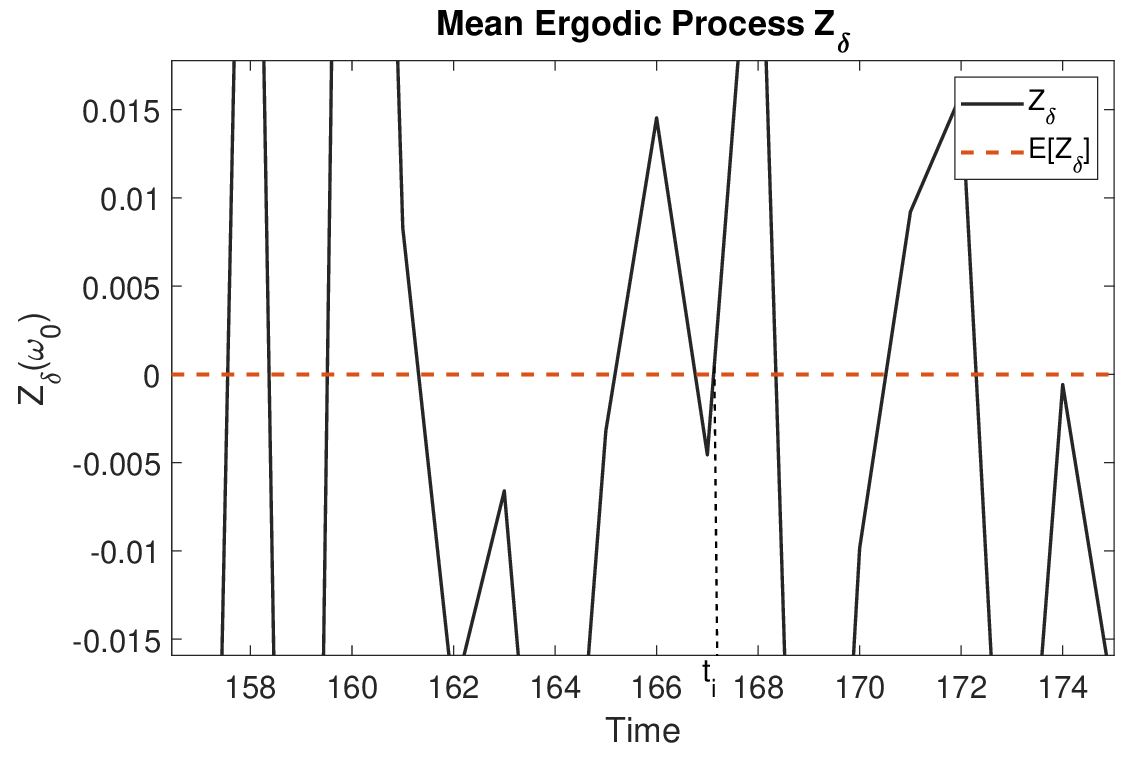}
		\caption{\small{A zoomed random sample path, $\omega_0$, of the process $Z_\delta$, for small time scales, with arbitrary recurrence time $t_i$ in the time interval $[166,168]$, with length $\delta_i=2$, in which the process $Z_\delta(\omega_0)$ returns to $\mathbb{E}[Z_{\delta}(\omega_0)]$. This random sample path is generated by the geometric Brownian motion process.}}
		\label{fig3}
	\end{center}
\end{figure}
In this section, we defined the concepts of the log-ergodic process and the ergodic maker operator and investigated their properties. In the next section, we present examples of log-ergodic processes usable for modeling financial markets with ergodic behavior in the mean.
\section{Log-Ergodic Processes in Mathematical Finance}\label{sec5}
In this section, we prove the log-ergodic property for the models widely used in mathematical finance.
\subsection{Mean reversion models}
\begin{proposition}
	Any stochastic process with mean-reverting property is mean-ergodic.
\end{proposition}
\begin{proof}
	 We know that every mean-reverting stochastic process is wide-sense stationary \cite{95}. Let $r_t(\cdot)$ be a stochastic process with mean reversion property and $\mathbb{E}[r_t]<\infty$. From \ref{mm}, it follows that
	\begin{equation*}
		\mathbb{E}[r_t]=\mathbb{E}[r_{t+\delta}], \quad \forall t,\delta>0.
	\end{equation*}
	Let $\tau_0>0$ be the time that the process $r_t$ meets its mean along the path $\omega_0$. According to the Poincaré recurrence theorem \cite{1} and theorem \ref{thr}, it can be written that:
	\begin{align*}
		&r_{\tau_0}(\omega_0)=\mathbb{E}[r_{\tau_0}(\omega_0)]\\
		&r_{\tau_0}(\omega_0)\cdot r_{\tau_0}(\omega_0) = r_{\tau_0}(\omega_0)\cdot \mathbb{E}[r_{\tau_0}(\omega_0)]\\
		&\mathbb{E}\big[\mathbb{E}[r_{\tau_0}^2(\omega_0)]\big]= \mathbb{E}[r_{\tau_0}(\omega_0)]\mathbb{E}\big[\mathbb{E}[r_{\tau_0}(\omega_0)]\big]\\
		&\mathbb{E}[r_{\tau_0}^2(\omega_0)]=\mathbb{E}^2[r_{\tau_0}(\omega_0)].
	\end{align*}
	Computing the covariance of the process at the time $\tau_0$ we obtain
	\begin{align}
		\mathbf{Cov}_{rr}(\tau_0)&=\mathbb{E}[r_{\tau_0}^2]-\big(\mathbb{E}[r_{\tau_0}]\big)^2<\infty\notag\\
		&=\mathbb{E}\big[\mathbb{E}[r_{\tau_0}]\mathbb{E}[r_{\tau_0}]\big]-\big(\mathbb{E}[r_{\tau_0}]\big)^2\notag\\
		&=\big(\mathbb{E}[r_{\tau_0}]\big)^2-\big(\mathbb{E}[r_{\tau_0}]\big)^2=0\notag.
	\end{align}
	Therefore, according to \ref{limer}, the process $r_t$ is mean-ergodic.
\end{proof}
\subsection{Main theorem}
In this subsection, we present the key theorem of the paper. In section \ref{sec6}, we use the results of this section to prove the theorem.
\begin{theorem}(Main theorem)\label{maintheo}
	Suppose that the price process, $S_t$, of an asset has the form:
	\begin{align*}
		S_t&=S_0e^{Y_t},\quad S_0=s,
	\end{align*}
	with
	\begin{align*}
		Y_t^\prime&=\ln(s)+Y_t=Y_0^\prime+\int_0^t \mu_{u,s}du+\int_0^t \sigma dW_u,\quad Y_0^\prime=\ln(s)+Y_0,\quad Y_0=0,\\ 
	\sigma&=f(V_t). 
	\end{align*}
	Where $\mu_{t,s}$ is an adapted function of $t$ and $s$, $V_t$ is an arbitrary random process, and $\sigma$ is an adapted function of the random process $V_t$ that satisfies the following conditions: $0<M_1\leq\sigma \leq M_2$ for some positive constants $M_1$ and $M_2$, and $\int_0^t \sigma_s^2ds<\infty$ for all $t>0$. Then, the process
	$S_t$ is partially ergodic.
\end{theorem}
In the remaining part of this section, we express some results that we use to prove the main theorem. 
\begin{theorem}\label{mainth}
	The stochastic process $Z_\delta$ defined by \ref{zt} is mean-ergodic. In other words, the positive stochastic process
	$X_t$ is partially ergodic.
\end{theorem}
\begin{proof}
	It follows from \cite{57,21,13} that $Z_\delta$ is a Markov process. $Z_\delta$ is stationary by \cite{58} and \ref{mm}. Therefore, $Z_\delta$ meets the requirements to be ergodic in the mean, as stated in \cite{59,35,68}. Hence, we first evaluate the expectation of the $Z_\delta$ process.
	\begin{align*}
		\mathbb{E}[Z_\delta]&=\mathbb{E}\Big[Z_0+\frac{1}{T^\beta}\int_0^\delta \sigma_s dW_s+\frac{W_T}{T^\beta}\int_0^\delta\mu_sds\Big]\\
		&=\mathbb{E}\Big[\frac{1}{T^\beta}\int_0^\delta\sigma_sdW_s\Big]+\mathbb{E}\Big[\frac{W_T}{T^\beta}\int_0^\delta\mu_sds\Big]=0.
	\end{align*}
Now we evaluate the time-average of $Z_\delta$.
\begin{align*}
	<Z>&=\lim_{T\rightarrow\infty}\frac{1}{T}\int_0^T Z_\delta d\delta\\
	&=\lim_{T\rightarrow\infty}\frac{1}{T}\int_0^T\Big[\frac{1}{T^\beta}\int_0^\delta\sigma_sdW_s+\frac{W_T}{T^\beta}\int_0^\delta \mu_sds\Big]d\delta\\
	&=\lim_{T\rightarrow\infty}\frac{1}{T}\int_0^T\Big[\frac{1}{T^\beta}\int_0^\delta\sigma_sdW_s\Big]d\delta+\lim_{T\rightarrow\infty}\frac{1}{T}\int_0^T\Big[\frac{W_T}{T^\beta}\int_0^\delta\mu_sds\Big]d\delta\\
	&=\lim_{T\rightarrow\infty}\frac{1}{T^{\beta+1}}\Big[\int_0^T \int_0^\delta\sigma_sdW_sd\delta\Big]+\lim_{T\rightarrow\infty}\frac{W_T}{T^{\beta+1}}\Big[\int_0^T \int_0^{\delta}\mu_sdsd\delta\Big].
\end{align*}
The first integral is zero since $d{W_s}d\delta=0$. Therefore,
\begin{align}\label{51}
	<Z>=\lim_{T\rightarrow\infty}\frac{W_T}{T^{\beta+1}}\Big[\int_0^T\int_0^\delta\mu_s ds d\delta \Big].
\end{align}
From \ref{ito} we have: $\int_0^\delta \lvert\mu_s\rvert ds <\infty$. Hence, the integral in \ref{51} is finite. From theorem \ref{IDV} it follows that $\lim_{T\rightarrow\infty}\frac{W_T}{T^{\beta+1}}=0$. Therefore, $<Z>=0$.
\end{proof}
As examples, the Ornstein–Uhlenbeck process \cite{84}, the geometric Brownian motion process \cite{13}, and the Jump-Diffusion process \cite{68} are positive partially ergodic processes.
\subsection{Log-ergodic Lévy processes}
Lévy processes have been studied in \cite{cont}, especially Cont studied the decomposition of exponential Lévy processes, which we will consider from the point of view of log-ergodicity. Lévy processes are well-behaved processes from the perspective of ergodic theory since their increments are stationary and independent. The independence of the increments implies that Lévy processes have Markov property \cite{86}. Therefore, any Lévy process satisfies the requirements for ergodicity.

Let $Y_t$ be a Lévy process and suppose that the distribution of $Y_t$ is parameterized by $(\eta,\sigma^2,\nu)$ \cite{86}. Using the Lévy-Itô theorem \cite{85,86}, we decompose $Y_t$ as
\begin{equation}
	Y_t=\eta t+\sigma W_t+J_t+M_t,\label{levy}
\end{equation} 
where $W_t$ is a standard Wiener process, for $t\geq0$, $\Delta Y_t=Y_t-Y_{t^-}$ is a Poisson process with intensity $\nu$,
$J_t=\sum_{s\leq t}\Delta Y_s\boldsymbol{1}_{\{\lvert \Delta Y_s \rvert >1\}},
$
and $M_t$ is a bounded martingale. Using the ergodic maker operator \ref{xi} for the process $Y_t$, we obtain:
\begin{align}
	Z_\delta=\frac{\sigma W_\delta}{T^{\beta}}+\frac{W_T}{T^{\beta}}\big[\delta\eta+J_\delta+M_\delta\big].\label{lev}
\end{align}
\begin{proposition}
Let $X_t=X_0e^{Y_t}$, with $X_0=x$, where $Y_t$ is a Lévy process. Then, the process $X_t$ is partially ergodic.
\end{proposition}
\begin{proof}
	The expectation of $Z_\delta$ defined by \ref{lev} is zero. Therefore, it suffices to prove that the time-average of $Z_\delta$ is zero.
	\begin{align}
		<Z>=&\lim_{T\rightarrow\infty}\frac{1}{T}\int_0^TZ_\delta d\delta\notag\\
		=&\lim_{T\rightarrow\infty}\frac{1}{T}\int_0^T\Big[\frac{\sigma W_\delta}{T^{\beta}}+\frac{W_T}{T^{\beta}}\big[\delta\eta+J_\delta+M_\delta\big]\Big] d\delta\notag\\
		=&\lim_{T\rightarrow\infty}\frac{1}{T}\int_0^T \frac{\sigma W_\delta}{T^{\beta}}d\delta+\lim_{T\rightarrow\infty}\frac{1}{T}\int_0^T\frac{\eta\delta W_T}{T^{\beta}}d\delta\notag\\
		&+\lim_{T\rightarrow\infty}\frac{1}{T}\int_0^T \frac{J_\delta W_T}{T^{\beta}}d\delta+\lim_{T\rightarrow\infty}\frac{1}{T}\int_0^T \frac{M_\delta W_T}{T^{\beta}}d\delta.\label{th}
	\end{align}
Now, using theorem \ref{IDV} we evaluate the integrals.\\
First integral:\\
The coefficient $\sigma$ is bounded. Therefore,
\begin{align*}
&\lim_{T\rightarrow\infty}\frac{1}{T}\int_0^T \frac{\sigma W_\delta}{T^{\beta}}d\delta
=\lim_{T\rightarrow\infty}\frac{\sigma}{T^{\beta+1}}\int_0^T W_\delta d\delta= \lim_{T\rightarrow\infty}\sigma\frac{\int_0^T W_\delta d\delta}{T^{\beta+1}}.
\end{align*}
Using \ref{IDV2} yields
\begin{align*}
	\sigma\lim_{T\rightarrow\infty}\frac{\int_0^T W_\delta d\delta}{T^{\beta+1}}=\sigma \times 0 =0.
\end{align*}
Second integral:
\begin{align*}
&\lim_{T\rightarrow\infty}\frac{1}{T}\int_0^T\frac{\eta\delta W_T}{T^{\beta}}d\delta=\lim_{T\rightarrow\infty}\frac{W_T\eta}{T^{\beta+1}}\int_0^T \delta d\delta
=\lim_{T\rightarrow\infty}\frac{W_T\eta}{T^{\beta+1}}\big(\frac{1}{2}\delta^2\big\rvert_0^T\big)=\lim_{T\rightarrow\infty}\frac{W_T\eta}{2T^{\beta-1}}=0.
\end{align*}
Third integral:
\begin{align*}
&\lim_{T\rightarrow\infty}\frac{1}{T}\int_0^T \frac{J_\delta W_T}{T^{\beta}}d\delta=\lim_{T\rightarrow\infty}\frac{W_T}{T^{\beta+1}}\int_0^T J_\delta d\delta=
\lim_{T\rightarrow\infty}\frac{W_T}{T^{\beta+1}}\int_0^T \sum_{s\leq \delta}\Delta Y_s\boldsymbol{1}_{\{\lvert \Delta Y_s \rvert >1\}} d\delta\\
=&\lim_{T\rightarrow\infty}\frac{W_T}{T^{\beta+1}}\sum_{s\leq \delta} \int_0^T[Y_s-Y_{s^-}]\boldsymbol{1}_{\{\lvert \Delta Y_s \rvert >1\}} d\delta\\
=&\lim_{T\rightarrow\infty}\frac{W_T}{T^{\beta+1}}\sum_{s\leq \delta} \int_0^T Y_s \boldsymbol{1}_{\{\lvert \Delta Y_s \rvert >1\}} d\delta - \lim_{T\rightarrow\infty}\frac{W_T}{T^{\beta+1}}\sum_{s\leq \delta} \int_0^T Y_{s^-} \boldsymbol{1}_{\{\lvert \Delta Y_s \rvert >1\}} d\delta\\
=&\lim_{T\rightarrow\infty}\frac{W_T}{T^{\beta+1}}\sum_{s\leq \delta} \int_0^T Y_s d\delta - \lim_{T\rightarrow\infty}\frac{W_T}{T^{\beta+1}}\sum_{s\leq \delta} \int_0^T Y_{s^-} d\delta\\
=&\lim_{T\rightarrow\infty}\frac{W_T}{T^{\beta}}\sum_{s\leq \delta}Y_s-\lim_{T\rightarrow\infty}\frac{W_T}{T^{\beta}}\sum_{s\leq \delta}Y_{s^-}=0.
\end{align*}
Fourth integral:
\begin{align*}
&\lim_{T\rightarrow\infty}\frac{1}{T}\int_0^T\frac{M_\delta W_T}{T^{\beta}}d\delta= \lim_{T\rightarrow\infty}\frac{W_T}{T^{\beta}}\int_0^T M_\delta d\delta.
\end{align*}
Since $M_\delta$ is bounded, the integral on the right-hand side is bounded. Therefore, using theorem \ref{IDV} yields
$$\lim_{T\rightarrow\infty}\frac{W_T}{T^{\beta}}\int_0^T M_\delta d\delta=0.$$
Hence, substituting the evaluated integrals in \ref{th}, we obtain $<Z>=0$. Therefore, the process $X_t$ is partially ergodic.
\end{proof}
As a result, any Poisson process, Itô process, and compound Poisson process is log-ergodic.
\subsection{Bounded processes}
\begin{proposition}
Let $Y_t$ be a non-negative bounded stochastic process. Then, the process $\xi_{\delta,W_\delta}^{\beta}[Y_t]$ is mean-ergodic.
\end{proposition}
\begin{proof}
Let $Y_t(\omega_0)$ be the path of the process $Y_t$ generated by $\omega_0\in \Omega$, and\\ $\xi_{\delta,W_\delta}^{\beta}[Y_t]=Z_\delta$. According to the definition of boundedness concerning $\omega_0$, we have:
\begin{equation*}
\lvert Y_t(\omega_0)\rvert \leq M,
\end{equation*}
for some positive number $M$. Now we write:
\begin{align*}
&Y_t(\omega_0)\leq \sup_{t\in[0,T]} Y_t(\omega_0)\\
\Rightarrow& \int_0^T \xi_{\delta,W_\delta}^\beta[Y_t(\omega_0)] d\delta\leq \int_0^T \sup_{t\in[0,T]} \xi_{\delta,W_\delta}^\beta[Y_t(\omega_0)] d\delta\\
\Rightarrow& \lim_{T\rightarrow \infty}\frac{1}{T}\int_0^T \underbrace{\xi_{\delta,W_\delta}^\beta[Y_t(\omega_0)]}_{Z_\delta(\omega_0)} d\delta\leq \lim_{T\rightarrow \infty}\frac{1}{T}\int_0^T\sup_{t\in [0,T]}\underbrace{\xi_{\delta,W_\delta}^\beta[Y_t(\omega_0)]}_{Z_\delta(\omega_0)} d\delta.
\end{align*}
From \cite{87}, for the covariance of the process $Z_\delta$, we consider the following relations:
\begin{equation*}
-\sqrt{{\mathbb{V}ar}^2\big[\xi_{\delta,W_\delta}^\beta[Y_t(\omega_0)]\big]}\leq \mathbf{Cov}_{zz}(\delta)\leq\sqrt{{\mathbb{V}ar}^2\big[\xi_{\delta,W_{\delta}}^\beta[Y_{t}(\omega_0)]\big]}.
\end{equation*} 
Thus, by definition \ref{limer} we get:
\begin{align}\label{vv}
&\lim_{T\rightarrow \infty}\frac{1}{T}\int_0^T(1-\frac{\delta}{T})\mathbf{Cov}_{zz}(\delta)d\delta\notag\\
\leq&\lim_{T\rightarrow \infty}\frac{1}{T}\int_0^T(1-\frac{\delta}{T})\big \lvert \mathbb{V}ar\big[\xi_{\delta,W_{\delta}}^\beta[Y_{t}(\omega_0)]\big]\big \rvert d\delta.
\end{align}
Let $\sup_{t\in[0,T]} Y_t(\omega)=M_t(\omega)$, for any $\omega\in \Omega$. We observe that:
\begin{align*}
0&\leq Y_t \leq M_t(\omega)\Rightarrow 0\leq \xi_{\delta,W_\delta}^\beta[Y_t]\leq \xi_{\delta,W_\delta}^\beta[M_t(\omega)]\\
\Rightarrow 0&\leq Z_\delta^2\leq (\xi_{\delta,W_\delta}^\beta[M_t(\omega)])^2\\
\Rightarrow 0&\leq \mathbb{E}[Z_\delta^2]\leq \mathbb{E}[(\xi_{\delta,W_\delta}^\beta[M_t(\omega)])^2].
\end{align*}
It follows from lemma \ref{mm} that:
\begin{equation}\label{varz}
0\leq\mathbb{V}ar[Z_\delta]\leq\mathbb{E}[(\xi_{\delta,W_\delta}^\beta[M_t(\omega)])^2].
\end{equation}
Now, using \ref{varz} in \ref{vv} yields
\begin{align*}
&\lim_{T\rightarrow \infty}\frac{1}{T}\int_0^T(1-\frac{\delta}{T})\mathbf{Cov}_{zz}(\delta)d\delta\\
\leq&\lim_{T\rightarrow \infty}\frac{1}{T}\int_0^T(1-\frac{\delta}{T})\mathbb{E}[(\xi_{\delta,W_\delta}^\beta[M_t(\omega)])^2]d\delta.
\end{align*}
Let $\mu_\xi(\delta)=\mathbb{E}[(\xi_{\delta,W_\delta}^\beta[M_t(\omega)])^2]$. According to \ref{xi}, we have $\mu_\xi(\delta)<\infty$. Thus,
\begin{align*}
\lim_{T\rightarrow \infty}\frac{1}{T}\int_0^T(1-\frac{\delta}{T})\mathbf{Cov}_{zz}(\delta)d\delta\leq
\lim_{T\rightarrow \infty}\frac{1}{T}\int_0^T(1-\frac{\delta}{T})\mu_\xi(\delta)d\delta.
\end{align*}
Evaluating the right-hand side yields 
\begin{align}
&\lim_{T\rightarrow \infty}\frac{1}{T}\int_0^T(1-\frac{\delta}{T})\mu_\xi(\delta)d\delta\notag\\
&=\lim_{T\rightarrow\infty}\big[\frac{1}{T}\int_0^T\mu_\xi(\delta)d\delta-\frac{1}{T^2}\int_0^T \delta\mu_\xi(\delta)d\delta\big].\label{re}
\end{align}
Since $\mu_\xi(\delta)$ is bounded, the first integral in \ref{re} approaches zero as $T\rightarrow \infty$. To evaluate the second integral, we proceed as follows:
\begin{align*}
	&\lim_{T\rightarrow \infty}\frac{1}{T}\int_0^T(1-\frac{\delta}{T})\mu_\xi(\delta)d\delta\notag\\
=&\lim_{T\rightarrow \infty}\frac{1}{T^2}\Big[\frac{1}{2}T^2\mu_\xi(T)-\frac{1}{2}\int_0^T 2\delta^2\mathbb{E}\big[\xi_{\delta,W_\delta}^\beta[M_t(\omega)]d\big(\xi_{\delta,W_\delta}^\beta[M_t(\omega)]\big)\big]\Big].
\end{align*}
$M_t(\omega)$ is independent of $W_t$. Therefore, $\mathbb{E}[\xi_{\delta,W_\delta}^\beta[M_t(\omega)]]=0$, according to \ref{xi}. 
Finally, from definition \ref{deferc} we have:
\begin{align*}
	\lim_{T\rightarrow \infty}\frac{1}{T^2}\frac{1}{2}T^2\mu_\xi(T)&=\lim_{T\rightarrow \infty}\frac{1}{2}\mathbb{E}[(\xi_{\delta,W_\delta}^\beta[M_t(\omega)])^2]=0.
\end{align*}
Therefore,
\begin{align*}
0\leq\lim_{T\rightarrow \infty}\frac{1}{T}\int_0^T(1-\frac{\delta}{T})\mathbf{Cov}_{zz}(\delta)d\delta\leq 0.
\end{align*}
Consequently, we get:
\begin{align*}
\lim_{T\rightarrow \infty}\frac{1}{T}\int_0^T(1-\frac{\delta}{T})\mathbf{Cov}_{zz}(\delta)d\delta=0.
\end{align*}
\end{proof}
\begin{example}
For $t>0$,
\item [1.] The process $Y_t=\sin(W_t)$ is mean-ergodic.
\item [2.] The process $Y_t=\gamma\sin(\mu t+\sigma W_t)$ where $\gamma$, $\mu$, and $\sigma$ are constants, is mean-ergodic.
\end{example}
\begin{solution}  \textbf{1.}
Using the Itô lemma we write:
\begin{align*}
&d\sin(W_t)=\cos(W_t)dW_t-\frac{1}{2}\sin(W_t)dt\\
&\mathbb{E}[\sin(W_t)]=-\frac{1}{2}\int_0^t\mathbb{E}[\sin(W_s)]ds\\
&d\mathbb{E}[\sin(W_t)]=-\frac{1}{2}\mathbb{E}[\sin(W_t)]dt\\
&\frac{d\mathbb{E}[\sin(W_t)]}{\mathbb{E}[\sin(W_t)]}=-\frac{1}{2}dt\\
&d\ln(\mathbb{E}[\sin(W_t)])=-\frac{1}{2}dt\Rightarrow \mathbb{E}[\sin(W_t)]=e^{\frac{-t}{2}}.
\end{align*}
Calculating the covariance of $\sin(W_t)$ we get:
\begin{align*}
\mathbf{Cov}_{yy}(\tau)=e^{-2\tau}-e^{\frac{-\tau}{2}}.
\end{align*}
Now, we calculate the limit in the definition \ref{limer} as follows: 
\begin{align*}
&\lim_{T\rightarrow \infty}\frac{1}{T}\int_0^T(1-\frac{\tau}{T})(e^{-2\tau}-e^{\frac{-\tau}{2}})d\tau\\
=&\lim_{T\rightarrow \infty}\frac{1}{T}\Big[-\frac{1}{2}e^{-2T}+2e^{\frac{-T}{2}}+\frac{1}{T}[\frac{(2T+1)e^{-2T}}{4}-2(T+2)e^{\frac{-T}{2}}]\Big]\\
=&\lim_{T\rightarrow \infty}\frac{1}{T}\Big[\frac{1}{4Te^{2T}}-\frac{4}{Te^{\frac{T}{2}}}\Big]=0.
\end{align*}
\end{solution}

\begin{solution}  \textbf{2.}
We have:
\begin{align*}
\mathbb{E}[Y_t]=&\mathbb{E}[\gamma\sin(\mu t+\sigma W_t)]\\
=&\int_{-\infty}^{\infty}\gamma\sin(\mu t+\sigma W_t)f_{ Y}(y)dy\\
=&\int_0^{2\pi} \gamma\sin(\mu t+\sigma W_t)\frac{1}{2\pi}dW_t\\
=&\frac{\gamma}{2\pi}\int_{0}^{2\pi}\sin(\mu t+\sigma W_t)dW_t=0.
\end{align*} 
Computing the time-average we obtain the following:
\begin{align*}
<Y>&=\lim_{T\rightarrow\infty}\frac{1}{T}\int_0^T\gamma\sin(\mu t+\sigma W_t)dt\\
&=\lim_{T\rightarrow\infty}\frac{\gamma}{T}\int_0^T\sin(\mu t+\sigma W_t)dt=0.
\end{align*}
Therefore, we have $<Y>=\mathbb{E}[Y_t]$. This implies that the process $Y_t$ is mean-ergodic.
\end{solution}
\begin{example}
It is proven in paper \cite{108} that a Markov chain that models a process confined to a bounded interval exhibits ergodic behavior while the process is constantly attracted to the center of the interval.
\end{example}

\section{Proof of Log-Ergodicity for Stochastic Volatility Models}\label{sec6}
In this section, we prove the main theorem of the paper, which we stated in section \ref{sec5}. First, we recall the statement of the theorem.
\begin{theorem}(Main theorem)
Suppose that the price process, $S_t$, of an asset has the form:
\begin{align*}
S_t&=S_0e^{Y_t},\quad S_0=s,
\end{align*}
with
\begin{align*}
Y_t^\prime&=\ln(s)+Y_t=Y_0^\prime+\int_0^t \mu_{u,s}du+\int_0^t \sigma dW_u,\quad Y_0^\prime=\ln(s)+Y_0,\quad Y_0=0,\\ 
\sigma&=f(V_t).
\end{align*}
Where $\mu_{t,s}$ is an adapted function of $t$ and $s$, $V_t$ is an arbitrary random process, and $\sigma$ is an adapted function of the random process $V_t$ that satisfies the following conditions: $0<M_1\leq\sigma \leq M_2$ for some positive constants $M_1$ and $M_2$, and $\int_0^t \sigma_s^2ds<\infty$ for all $t>0$. Then, the process $S_t$ is partially ergodic.
\end{theorem}
\begin{proof}
We write:
\begin{align*}
Y_t^\prime&=\ln(s)+\int_0^t \mu_{u,s} du+\int_0^t f(V_u) dW_u,\\
Z_\delta&=\xi_{\delta,W_\delta}^\beta[Y_t^\prime]=\frac{W_T\int_0^\delta \mu_{u,s}du}{T^{\beta}}+\frac{\int_0^\delta f(V_u)dW_u}{T^{\beta}},\quad Z_0=0.
\end{align*}
Now we evaluate the covariance of $Z_\delta$:
\begin{align*}
\mathbf{Cov}_{zz}(\delta)=\frac{\mathbb{E}\big[(\int_0^\delta \mu_{u,s}du)^2\big]}{T^{2\beta-1}}+\frac{\mathbb{E}\big[(\int_0^\delta f(V_u)dW_u)^2\big]}{T^{2\beta}}.
\end{align*}
Next, we prove \ref{limer} holds.
\begin{align*}
\overline{<Z>}=&\lim_{T\rightarrow \infty}\frac{1}{T}\Big[\int_0^T(1-\frac{\delta}{T})\frac{\mathbb{E}\big[(\int_0^\delta\mu_{u,s}du)^2\big]}{T^{2\beta-1}}d\delta\notag\\
&+\int_0^T(1-\frac{\delta}{T})\frac{\mathbb{E}\big[\int_0^\delta f^2(V_u)du\big]}{T^{2\beta}} d\delta\Big].
\end{align*}
It follows from theorem \ref{mainth} that the first integral approaches zero as $T\rightarrow \infty$. Therefore, it suffices to prove
\begin{align*}
&\lim_{T\rightarrow \infty}\frac{1}{T}\big[\int_0^T(1-\frac{\delta}{T})\frac{\mathbb{E}[\int_0^\delta f^2(V_u)du]}{T^{2\beta}} d\delta\big]=0.
\end{align*}
The volatility term at any time interval of length $\delta=t-s$, for all $t,s>0$ and $t\neq s$, is bounded \cite{13}. Therefore, there are positive numbers $M_1$, and $M_2$ such that:
$$0<M_1\leq f(V_t) \leq M_2, \quad a.s.$$
Thus,
\begin{align*}
M_1^2&\leq f^2(V_t)\leq M_2^2,\\
\int_0^\delta M_1^2 du&\leq \int_0^\delta f^2(V_u)du\leq \int_0^\delta M_2^2 du,\\
\mathbb{E}[M_1^2\delta]&\leq \mathbb{E}[\int_0^\delta f^2(V_u)du]\leq \mathbb{E}[M_2^2\delta],\\
\frac{M_1^2\delta}{T^{2\beta}}&\leq \frac{\mathbb{E}[\int_0^\delta f^2(V_u)du]}{T^{2\beta}}\leq \frac{M_2^2\delta}{T^{2\beta}}.
\end{align*}
As a result, we get:
\begin{align}
\int_0^T\frac{M_1^2\delta}{T^{2\beta}}d\delta&\leq \int_0^T\frac{\mathbb{E}[\int_0^\delta f^2(V_u)du]}{T^{2\beta}}d\delta\leq \int_0^T\frac{M_2^2\delta}{T^{2\beta}}d\delta,\notag\\
\frac{M_1^2}{2T^{2\beta-2}}&\leq \int_0^T\frac{\mathbb{E}[\int_0^\delta f^2(V_u)du]}{T^{2\beta}}d\delta\leq \frac{M_2^2}{2T^{2\beta-2}},\notag\\
\lim_{T\rightarrow \infty}\frac{M_1^2}{2T^{2\beta-2}}&\leq \lim_{T\rightarrow \infty}\frac{1}{T}\int_0^T\frac{\mathbb{E}[\int_0^\delta f^2(V_u)du]}{T^{2\beta}}d\delta\leq \lim_{T\rightarrow \infty}\frac{M_2^2}{2T^{2\beta-2}},\notag\\
\Rightarrow\quad 0&\leq \lim_{T\rightarrow \infty}\frac{1}{T}\int_0^T\frac{\mathbb{E}[\int_0^\delta f^2(V_u)du]}{T^{2\beta}}d\delta\leq 0.\label{I1}
\end{align}
Also we have:
\begin{align}
\frac{M_1^2\delta^2}{T^{2\beta+1}}&\leq\frac{\mathbb{E}[\int_0^\delta f^2(V_u)du]\delta}{T^{2\beta+1}}\leq\frac{M_2^2\delta^2}{T^{2\beta+1}},\notag\\
\frac{1}{T}\int_0^T\frac{M_1^2\delta^2}{T^{2\beta+1}}d\delta&\leq \int_0^T\frac{1}{T^{2\beta+2}}\mathbb{E}[\int_0^\delta f^2(V_u)du]\delta d\delta\leq \frac{1}{T}\int_0^T\frac{M_2^2\delta^2}{T^{2\beta+1}}d\delta,\notag\\
\lim_{T\rightarrow\infty}\frac{M_1^2}{3T^{2\beta-1}}&\leq \lim_{T\rightarrow \infty}\frac{1}{T^2}\int_0^T\frac{\mathbb{E}[\int_0^\delta f^2(V_u)du]\delta}{T^{2\beta}}d\delta\leq \lim_{T\rightarrow \infty}\frac{M_2^2}{3T^{2\beta-1}},\notag\\
\Rightarrow\quad 0&\leq \lim_{T\rightarrow \infty}\frac{1}{T^2}\int_0^T\frac{\mathbb{E}[\int_0^\delta f^2(V_u)du]\delta}{T^{2\beta}}d\delta\leq 0.\label{I2}
\end{align}
Now it follows from \ref{I1} and \ref{I2} that $\overline{<Z>}=0$.
\end{proof}

\section{Application of Log-Ergodic Processes in Mathematical Finance}\label{sec7}
The main benefit of using the log-ergodic processes is the substitution of time-averaging with expectation in computations in the long run. We will use the results of this paper to model leveraged futures trading by estimating mean reversion time intervals in subsequent studies. Reducing the randomness of a financial model reduces the risk of trading and allows one to enter or leave a trading position with a lower risk. 

In the following, we will study the behavior of the log-ergodic processes using empirical data and express a novel version of the Black-Scholes partial differential equation by providing an example concerning the simulation of the mean-ergodic process $Z_\delta$.
\begin{example}
Consider the empirical data of Tesla stock price from December 14, 2011, to December 14, 2021. We extracted the data from the Yahoo Finance \footnote{\url{https://finance.yahoo.com}} website. Considering the stock price process, $S_t$, follows the geometric Brownian motion, we write:
\begin{align*}
S_t&=S_0\exp\big\{(\mu-\frac{1}{2}\sigma^2)t+\sigma W_t\big\},\quad S_0=s.\\
Y_t^\prime&=\ln(S_t)=Y_0^\prime+(\mu-\frac{1}{2}\sigma^2)t+\sigma W_t,\quad Y_0^\prime=\ln(s).
\end{align*}
Where $\mu$ and $\sigma$ are constants, and $W_t$ is a standard Wiener process. Using the ergodic maker operator for any time interval $\delta$, we have:
$$Z_\delta=Z_0+\frac{(\mu-\frac{1}{2}\sigma^2)\delta W_T}{T^{\beta}}+\frac{\sigma W_\delta}{T^{\beta}},\quad Z_0=0.$$
A random path of $Z_\delta$ for the data of the Tesla for $\beta=2$ is shown in Figure \ref{fig2}.
\begin{figure}[H]
\begin{center}
\includegraphics[scale=0.55]{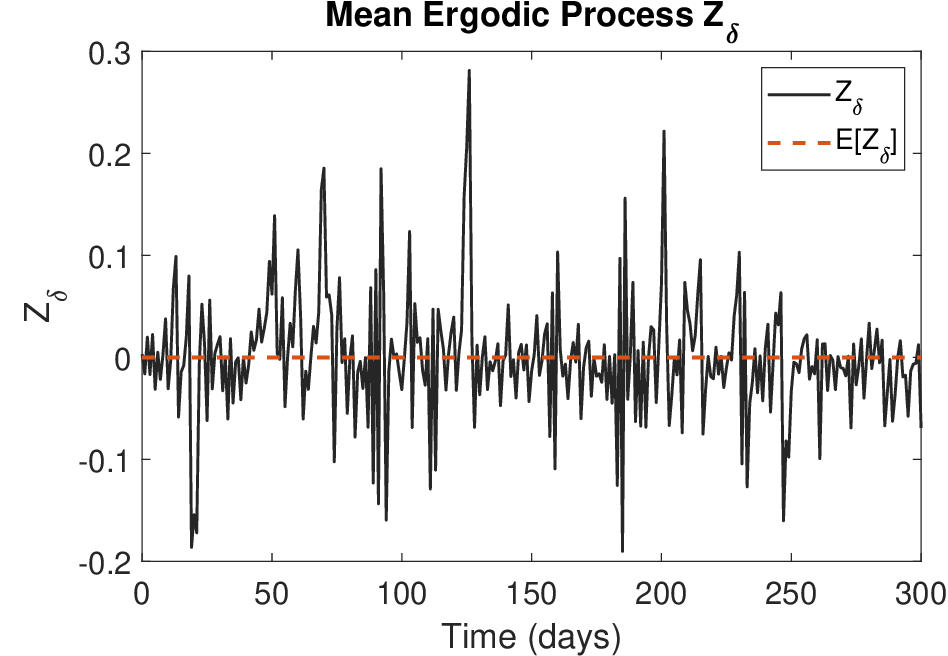}
\caption{\small{A random path of the process $Z_\delta$ for the Tesla price data for $\beta=2$ within a 300-day time frame.}}
\label{fig2}
\end{center}
\end{figure}
\end{example}
Let $q=\mu-\frac{1}{2}\sigma^2$ and $\delta=T-0=T$. Using the Itô lemma, we have:
\begin{align}
dZ_T&=\frac{1}{2}\big[0\big]dT+\Big[\frac{q}{T^{\beta-1}}+\frac{\sigma}{T^{\beta}}\Big]dW_T-\Big[\frac{(\beta-1)qW_T}{T^{\beta}}+\frac{\beta\sigma W_T}{T^{\beta+1}}\Big]dT,\notag\\
&=\Big[\frac{(1-\beta)qW_T}{T^{\beta}}-\frac{\beta\sigma W_T}{T^{\beta+1}}\Big]dT+\Big[\frac{q}{T^{\beta-1}}+\frac{\sigma}{T^{\beta}}\Big]dW_T,\notag\\
&=\Big[\frac{qW_T}{T^\beta}-\frac{\beta}{T}\big[\frac{qW_T}{T^{\beta-1}}+\frac{\sigma W_T}{T^\beta}\big]\Big]dT+\Big[\frac{q}{T^{\beta-1}}+\frac{\sigma}{T^\beta}\Big]dW_T\notag\\
\Rightarrow dZ_T&=\underbrace{\Big[\frac{qW_T}{T^\beta}-\frac{\beta}{T}Z_T\Big]}_{A(T,W_T)}dT+\underbrace{\Big[\frac{q}{T^{\beta-1}}+\frac{\sigma}{T^\beta}\Big]}_{B_T}dW_T\notag\\
dZ_T&=A(T,W_T)dT+B_T dW_T.\notag
\end{align}
Consequently, since we considered $\delta=T-0=T$, we can write:
\begin{align}
	dZ_\delta&=A(\delta,W_\delta)d\delta+B_\delta dW_\delta.
	\end{align}
According to the above assumptions, we have the following result:
\begin{proposition}(Ergodic Partial differential equation of Black–Scholes)
Under the assumptions of the Black-Scholes model, the European call option price, $C(Z_\delta,\delta)$, relative to the stock price variation $Z_\delta=z$, concerning short rate $r$, inhibition degree parameter $\beta$, and the strike price $K$ satisfies in the following partial differential equation.
\begin{align}
&\frac{\partial C}{\partial \delta}+rz\frac{\partial C}{\partial z}+\frac{1}{2}B_\delta^2\frac{\partial^2 C}{\partial z^2}-rC=0,\label{BSP}\\
&\text{for} \quad 0<\lvert z\rvert<\infty,\quad 0<\delta<\delta_T=T-0,\notag\\
&\text{where}\quad B_\delta=\frac{q}{\delta^{\beta-1}}+\frac{\sigma}{\delta^{\beta}}, \quad q=\mu-\frac{1}{2}\sigma^2,\notag
\end{align}
together with initial conditions $C(0,\delta)=0$ and $C(z,\delta_T)=(\lvert z\rvert-\ln(K))^+$.
\end{proposition}
\begin{proof}
We think of $Z_\delta$ as the process of the variations of the price of a traded stock and form a risk hedging basket, including $x$ shares with the price variation $z$ and one unit of call option with a sell position \cite{13}. For the price $V_\delta:=V(z,\delta)$ of this basket, we have:
\begin{align}
V_\delta&=-C(z,\delta_T)+xz,\notag \\
dV_\delta&=-dC(z,\delta_T)+xdz.\label{V}
\end{align}
We set $C:=C(z,\delta_T)$. Substituting the dynamics of $z$ in \ref{V} and using Itô lemma yields
\begin{align}
dV_\delta=&-\frac{\partial C}{\partial \delta}d\delta-\frac{\partial C}{\partial z}A_z(\delta,W_\delta)d\delta-\frac{\partial C}{\partial z}B_\delta dW_\delta-\frac{1}{2}\frac{\partial^2 C}{\partial z^2}B_\delta^2d\delta\notag\\
&+xA_z(\delta,W_\delta)d\delta+xB_\delta dW_\delta,\notag\\
dV_\delta=&-\Big[\frac{\partial C}{\partial \delta}-xA_z(\delta,W_\delta)+\frac{\partial C}{\partial z}A_z(\delta,W_\delta)+\frac{1}{2}\frac{\partial^2 C}{\partial z^2}B_\delta^2\Big]d\delta\notag\\
&+\Big[xB_\delta-\frac{\partial C}{\partial z}B_\delta\Big]dW_\delta.\label{V1}
\end{align}
For the portfolio to be risk-free, the coefficient of $dW_t$ must be zero. We therefore have $x=\frac{\partial C}{\partial z}$. Substituting the value of $x$ in \ref{V1}, we reach the following:
\begin{align*}
dV_\delta=&-\Big[\frac{\partial C}{\partial \delta}+\frac{\partial C}{\partial z}A_z(\delta,W_\delta)-\frac{\partial C}{\partial z}A_z(\delta,W_\delta)+\frac{1}{2}\frac{\partial^2 C}{\partial z^2}B_\delta^2\Big]d\delta\\
dV_\delta=&-\Big[\frac{\partial C}{\partial \delta}+\frac{1}{2}\frac{\partial^2 C}{\partial z^2}B_\delta^2\Big]d\delta
\end{align*}
On the other hand, by the absence of arbitrage, we have: $dV_\delta=rV_\delta d\delta$ \cite{13}. Therefore,
\begin{align}
&r\Big[-C+\frac{\partial C}{\partial z}z\Big]d\delta\notag=-\Big[\frac{\partial C}{\partial \delta}+\frac{1}{2}\frac{\partial^2 C}{\partial z^2}B_\delta^2\Big]d\delta\notag.
\end{align}
Simplifying, we get the equation:
\begin{align*}
&\frac{\partial C}{\partial \delta}+rz\frac{\partial C}{\partial z}+\frac{1}{2}B_\delta^2\frac{\partial^2 C}{\partial z^2}-rC=0.
\end{align*}
Where
\begin{align*}
 B_\delta=\frac{q}{\delta^{\beta-1}}+\frac{\sigma}{\delta^{\beta}}, \quad q=\mu-\frac{1}{2}\sigma^2.
\end{align*}
The European call option is exercised when $\lvert z\rvert>\ln(K)$. Therefore, the final condition for $\delta=\delta_T$ is $C(z,\delta_T)=\max\big[\lvert z\rvert-\ln(K),0\big]$.
Also, regarding the boundary conditions we have:
\begin{align*}
&C(0,\delta)=0, \quad 0<\delta<\delta_T=T,\\
&C(z,\delta)\sim\lvert z\rvert-\ln(K)e^{-r(\delta_T-\delta)}, \quad \text{as} \quad \lvert z\rvert\rightarrow \infty.
\end{align*}
\end{proof}
 We use the ergodic maker operator (EMO), a mathematical tool transforming a non-ergodic process into an ergodic one, to incorporate the inhibition degree parameter into the financial models ( e.g., the partial differential equation \ref{BSP} ). The inhibition degree parameter describes market imperfections or constraints that affect the price dynamics and may vary over time depending on external shocks or events ( e.g., natural disasters, black swans, wars, elections, and more). We use the log-ergodic processes to model financial markets because they allow us to study the behavior of the market participants from the perspective of the invisible hands that govern the market equilibrium.
\section{Empirical Data Analysis}\label{sec8}
In this section, we present the empirical data analysis and express the results of our study. We use quantitative methods to test our hypotheses and predictions derived from the log­-ergodicity theory. We use the statistical package IBM SPSS Statistics and Matlab to perform the analysis. 
\subsection{Data Description} 
We use three datasets for our empirical study. The first dataset contains the daily closing prices of Tesla (TSLA) stock from December 14, 2011, to December 14, 2021. The second dataset contains the daily closing prices of Apple (AAPL) stock from December 14, 2011, to December 14, 2021. The third dataset contains the daily closing prices of Microsoft (MSFT) stock from December 14, 2011, to December 14, 2021. We obtained the data from Yahoo Finance\footnote{\url{https://finance.yahoo.com}}.
 
We transform the price data into log-returns by taking the natural logarithm of the ratio of consecutive prices. We then apply the ergodic maker operator to the log-return data with different values of the inhibition degree parameter $\beta$. We obtain the log-ergodic returns by multiplying the log-returns by $\beta$ and adding a constant term $\alpha$ that ensures the positivity of the resulting process. We choose $\alpha$ from the range $[0, 0.1]$ with a step size of $0.01$ and $\beta$ from the range $[1.6, 2]$ with a step size of $0.1$. We generate log-ergodic processes for each original price process. 
\subsection{Data Analysis Methods} 
We use three methods to analyze the data: descriptive statistics, correlation analysis, and regression analysis. 
 \begin{itemize}
	\item [1.]
Descriptive statistics: We compute the mean, standard deviation, skewness, and kurtosis of each log-return and log-ergodic return series. Additionally, we plot the histograms of the distributions of each series and compare the descriptive statistics of the original and transformed processes to examine how does the ergodic maker operator affect the properties of the price dynamics.
\item[2.]
Correlation analysis: We compute the Pearson correlation coefficients between log-return and log-ergodic return series. Also, we plot the correlation matrices and the heatmaps of the correlation coefficients and compare the correlation coefficients of the original and transformed processes to examine how the ergodic maker operator affects the dependence structure of the price movements. 
\item [3.]
Regression analysis: We use the Fourier regression model to test our hypotheses and predictions about the effects of log-ergodicity on pricing contingent claims and studying market restrictions. We report the regression coefficients, the sum of squares error, the root mean squared error, R-squared, and adjusted R-squared for each model and plot the scatterplots and regression lines for each model.
\end{itemize}
\subsection{Data Analysis Results}
We present the results of our data analysis in this subsection. We summarize the main findings and discuss their implications for our research questions. 
\subsubsection{Descriptive Statistics}
The descriptive statistics of the log-return and log-ergodic return series are shown in Table \ref{t1} and Table \ref{t2}, respectively. Figure \ref{F1} shows the histograms of the distributions of the series.

\begin{table}[ht]
	\centering
	\caption{Descriptive statistics of log-return series.}\label{t1}
\begin{tabular}{|l|c|c|c|c|}
	\hline
	Stock & Mean & Standard deviation & Skewness & Kurtosis \\
	\hline
	TSLA & 0.0017 & 0.0347 & -0.0523 & 8.1688 \\
	AAPL &  0.0009 & 0.0164 & -0.0312 & 4.0357 \\
	MSFT & 0.0008 & 0.0145 & -0.0213 & 3.9012   \\
	\hline
\end{tabular}
\end{table}
\begin{table}[ht]
	\centering
	\caption{Descriptive statistics of log-ergodic return series.}\label{t2}
	\begin{tabular}{|l|c|c|c|c|}
		\hline
		Stock ($\beta$,$\alpha$) & Mean & Standard deviation & Skewness & Kurtosis \\
		\hline
		TSLA (1.6, 0) & 0.00085 & 0.01735 & -0.0523 & 8.1688 \\
		TSLA (1.6, 0.01) & 0.01085 & 0.01735 & -0.0523 & 8.1688  \\
		\vdots & \vdots& \vdots& \vdots& \vdots\\
		TSLA (2, 0) & 0.0034 & 0.0694 & -0.0523 & 8.1688 \\
		TSLA (2, 0.01) & 0.0134 & 0.0694 & -0.0523 & 8.1688 \\
	    \vdots & \vdots& \vdots& \vdots& \vdots\\
		AAPL (1.6, 0) & -3.1434e-04 & 0.0065 & -0.0312 & 4.0357 \\
		AAPL (1.6, 0.01)& 1.5451e-04 & 0.0106 & -0.0312 & 4.0357  \\
		\vdots & \vdots& \vdots& \vdots& \vdots\\
		AAPL (2, 0) & -5.9962e-04 & 0.0065 & -0.0312 & 4.0357 \\
		AAPL (2, 0.01) & 3.4578e-04 & 0.0095 & -0.0312 & 4.0357
		  \\
		  \vdots & \vdots& \vdots& \vdots& \vdots\\
		MSFT (1.6, 0) & 0.0004 & 0.00725 & -0.0213 & 3.9012 \\
		MSFT (1.6, 0.01) & 0.0104 & 0.00725 & -0.0213 & 3.9012  \\
		\vdots & \vdots& \vdots& \vdots& \vdots\\
		MSFT (2, 0) & 0.0016 & 0.029 & -0.0213 & 3.9012 \\
		MSFT (2, 0.01)& 0.0116 & 0.029 & -0.0213 & 3.9012  \\
		\vdots & \vdots& \vdots& \vdots& \vdots\\
		\hline
	\end{tabular}
\end{table}
\begin{figure}[H]
	\centering
	 \begin{subfigure}{0.6\textwidth}
		\centering
		\includegraphics[scale=0.52]{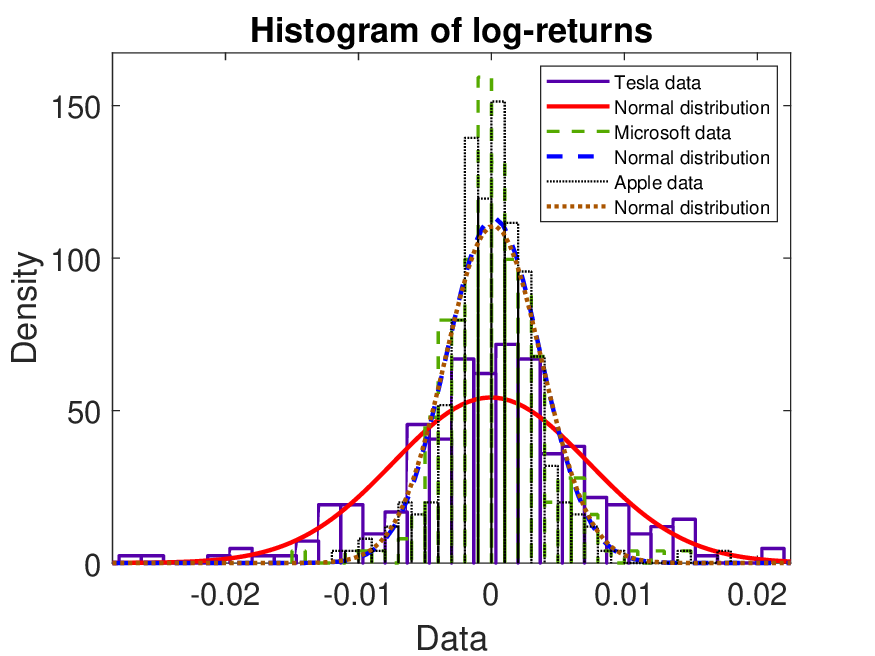}
		\label{F1a}
	\end{subfigure}
	\hfill
	\begin{subfigure}{0.6\textwidth}
		\centering
		\includegraphics[scale=0.52]{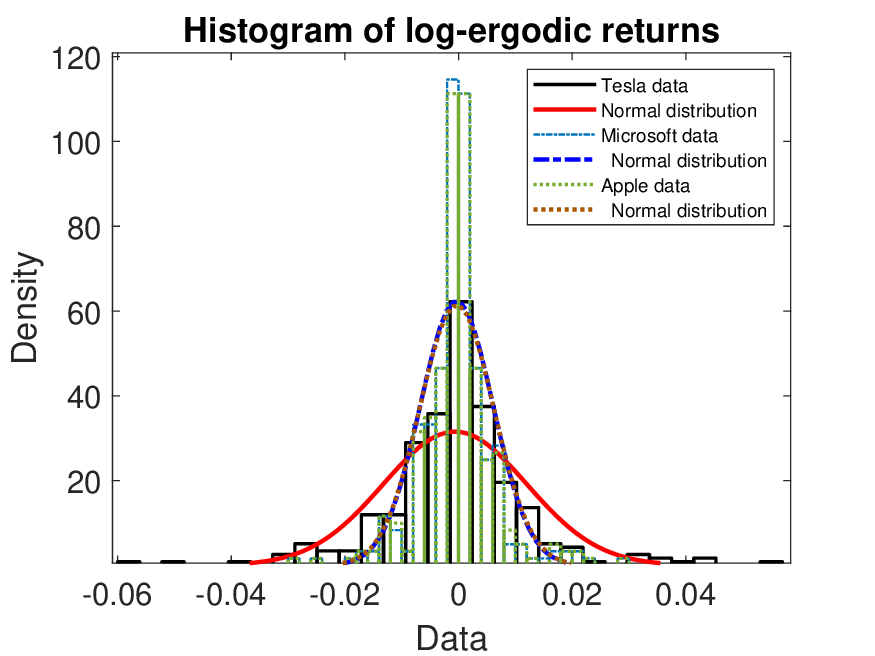}
		\label{F1b}
	\end{subfigure}
	\caption{Histograms of log-return and log-ergodic return series.}
	\label{F1}
\end{figure}
From the descriptive statistics, we can observe the following patterns: 
\begin{itemize}
\item[1.]The mean of the log-ergodic return series increases with $\alpha$ and $\beta$. This result is consistent with the definition of the ergodic maker operator, which adds a constant term $\alpha$ to the log-returns and scales them using $\beta$.
\item[2.]The standard deviation of the log-ergodic return series increases with $\beta$ and decreases with $\alpha$. This result is also consistent with the definition of the ergodic maker operator, which scales the variance of the log-returns by $\beta^2$ and reduces the volatility by adding a constant term $\alpha$.
\item[3.]	The skewness and kurtosis of the log-ergodic return series are equal to those of the log-return series for each original price process because the ergodic maker operator does not change the shape of the distribution of the log-returns but only shifts and stretches it.
\item[4.]	The histograms show that the distributions of log-return and log-ergodic return series are approximately symmetric and bell-shaped, with some outliers and heavy tails. 
\end{itemize}
\subsubsection{Correlation Analysis}
The correlation coefficients between pairs of log-return and log-ergodic return series are shown in Table \ref{t3} and Table \ref{t4}, respectively. Figure \ref{F2} shows the heatmaps of the correlation coefficients. 
\begin{table}[ht]
	\centering
	\caption{Correlation coefficients between log-return series.}\label{t3}
	\begin{tabular}{|l|c|c|c|}
		\hline
		Series & TSLA & AAPL & MSFT  \\
		\hline
		TSLA & 1 & 0.3123 & 0.2412 \\
		\hline
		AAPL & 0.3123 & 1 & 0.5321   \\
		\hline
		MSFT & 0.2412 & 0.5321 & 1  \\
		\hline
\end{tabular}
\end{table}
\begin{table}[ht]
	\centering
	\caption{Correlation coefficients between log-ergodic return series.}\label{t4}
		\hspace*{-1cm}
	\begin{tabular}{|l|c|c|c|c|c|c|}
		\hline
		Series($\beta$,$\alpha$) & TSLA(1.6,0) & TSLA(2,0) &  AAPL(1.6,0) & AAPL(2,0) & MSFT(1.6,0) & MSFT(2,0) \\
		\hline
		TSLA(1.6,0) & 1 & -1 &0.3123 & -0.3123 & 0.2412 & -0.2412 \\
		\hline
		TSLA(2,0) & -1 & 1 & -0.3123 & 0.3123  & -0.2412 & 0.2412 \\
		\hline
		AAPL(1.6,0) & 0.3123 & -0.3123 & 1 & -1  & 0.5321 & -0.5321 \\
		\hline
		AAPL(2,0) & -0.3123 & -0.3123 & -1 & 1  & -0.5321 & 0.5321 \\
		\hline
		MSFT(1.6,0) & 0.2412 & -0.2412 & 0.5321 & -0.5321 & 1 & -1 \\
		\hline
		MSFT(2,0) & -0.2412 & 0.2412 & -0.5321 & 0.5321 & -1 & 1 \\
		\hline
	\end{tabular}
\end{table}
\begin{figure}[H]
	\centering
	\hspace*{-0.2cm}
	\begin{subfigure}{0.4\textwidth}
		\centering
		\hspace*{-0.8cm}
		\vspace*{-0.7cm}
		\includegraphics[scale=0.4]{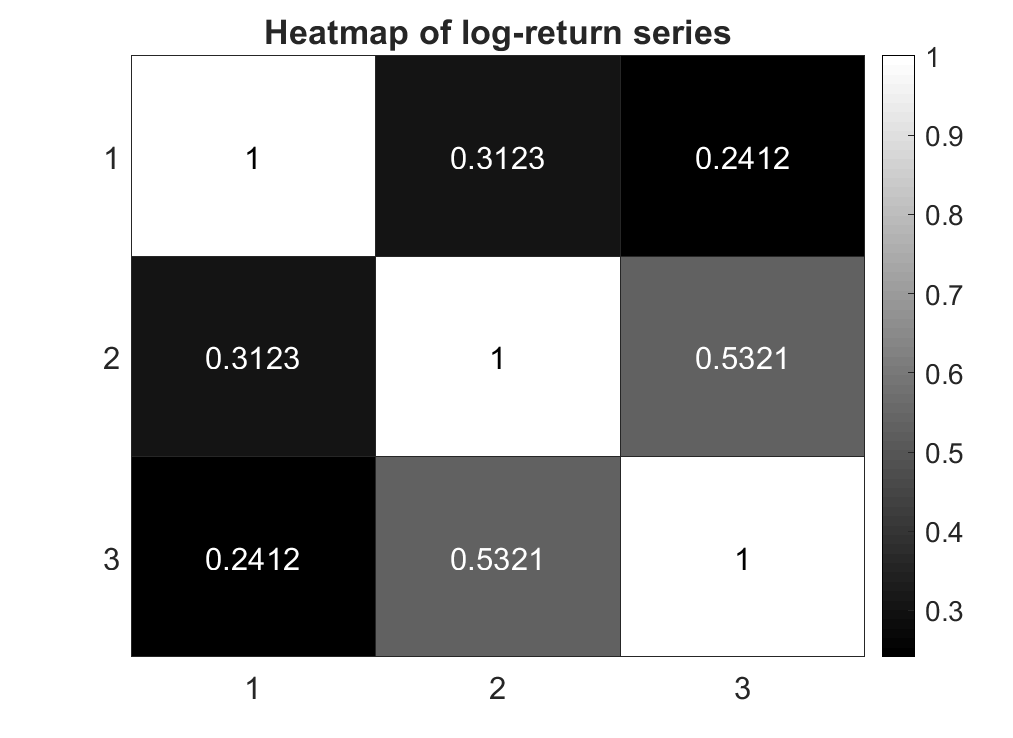}
		\label{F2a}
	\end{subfigure}
	\hfill
	\begin{subfigure}{0.6\textwidth}
		\centering
		\hspace*{1cm}
		\includegraphics[scale=0.4]{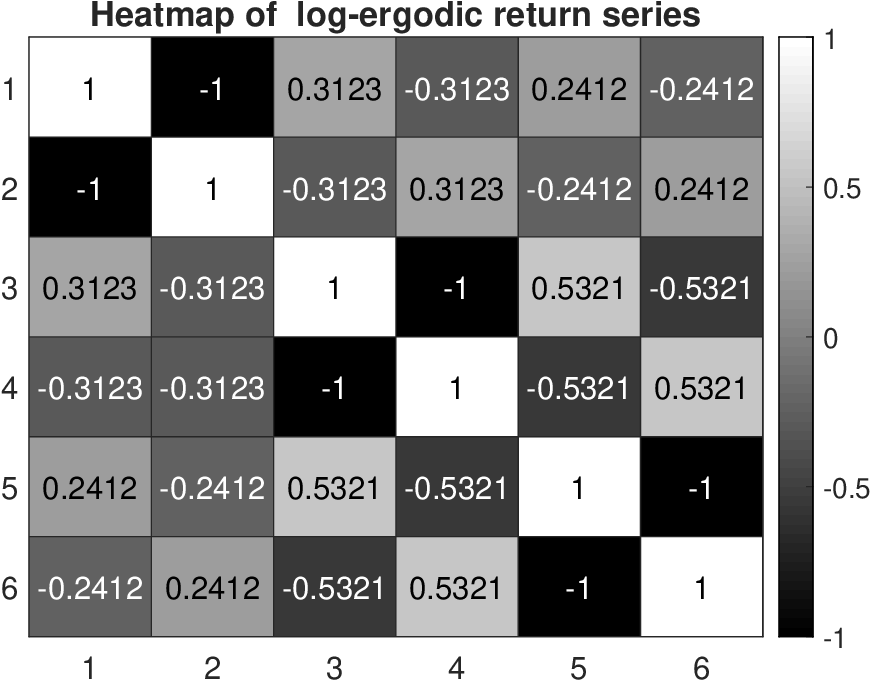}
		\label{F2b}
	\end{subfigure}
	\caption{Heatmaps of the correlation coefficients of log-return and log-ergodic return series }
	\label{F2}
\end{figure}
The correlation analysis shows that the log-return series are positively correlated with each other, indicating that the price movements of different stocks are under the influence of common factors. The log-ergodic return series are negatively correlated with each other, indicating that the ergodic maker operator reduces the dependence structure of the price movements. The log-ergodic return series are also negatively correlated with their corresponding log-return series, meaning that the ergodic maker operator changes the direction of the relationship between the original and transformed processes. 
These results suggest that 
the log-ergodic models can capture and model the ergodic behavior of the risky assets ( hidden from market participants ) and have advantages over other models, such as the geometric Brownian motion.
\subsubsection{Regression Analysis}
To perform the regression analysis, we use the Fourier regression model with $8$ terms:
\begin{align*}
	z(x_i)=& 
	a_0+a_1\cos(x_i w)+b_1\sin(x_i w)+a_2\cos(2x_iw) + b_2\sin(2x_iw) +\\ &a_3\cos(3x_iw)+b_3\sin(3x_iw)+a_4\cos(4x_iw)+b_4\sin(4x_iw)+\\ 
	&a_5\cos(5x_iw)+b_5\sin(5x_iw)+a_6\cos(6x_iw)+b_6\sin(6x_iw)+\\
	&a_7\cos(7x_iw)+b_7\sin(7x_iw)+a_8\cos(8x_iw)+b_8\sin(8x_iw)+\epsilon_i,
\end{align*}
where $z(x_i)$ is the log-ergodic return series of stock $i$, $x_i$ is the log-return series of stock $i$, $a_i$ and $b_i$ are the regression coefficients, and $\epsilon_i$ is the error term. We report the results in Table \ref{t5}. Figure \ref{F6} shows the plots of the analysis of the series.

\begin{table}[hb]
	\centering
	\caption{Regression results of log-ergodic return series on log-return series.}\label{t5}
	\begin{tabular}{|l|c|c|c|c|c|c|}
		\hline
		Stock & SSE & R-squared & adj R-squared & RMSE & p-value & $\#$ Coeff \\
		\hline
		TSLA & 0.1617 & 0.4537 & 0.4139 & 0.02635 & $<0.01$ &18 \\
		\hline
		AAPL & 0.04041 & 0.4747 & 0.4363 & 0.01317  & $<0.01$ & 18 \\
		\hline
		MSFT & 0.04232 & 0.4682 & 0.4229 & 0.01348 & $<0.01$ & 18 \\
		\hline
	\end{tabular}
\end{table}
\begin{table}[hb]
	\centering
	\caption{Calculated coefficients for Regression results (with $95\%$ confidence bounds)}\label{t6}
	\begin{tabular}{|l|c|c|c|}
		\hline
		\backslashbox{\tabular{@{}l@{}}Coeff-\\icients\endtabular}{	\vspace*{-0.35cm}Stock}
		 & TSLA & AAPL & MSFT \\
		\hline
		a0 & -0.001263  & -0.0006846 & -0.0007385 \\
		\hline
		a1 &  0.0004105 & 0.0003747 & 0.0001122  \\
		\hline
		b1 & -1.332e-05 & 6.874e-05 &  0.0006728 \\
		\hline
		a2 & -0.0005494 & 9.148e-05 & 0.0001917  \\
		\hline
		b2 & -0.0001405 & -0.0001274 &  7.042e-05\\
		\hline
		a3 &  0.0007418 & -7.123e-05 & -0.000891  \\
		\hline
		b3 & -2.055e-05 & 0.0004703 & -0.0005245  \\
		\hline
		a4 &  0.00082   & -0.0005152 &  7.257e-05  \\
		\hline
		b4 &  0.0003367 & -0.0005578  & -0.0001921  \\
		\hline
		a5 & -0.000403  & 0.0001929 &  0.0004565  \\
		\hline
		b5 &  0.00157   & 0.000778  & -0.0001015  \\
		\hline
		a6 & -0.0004864 & -0.0001518 & -0.0001006 \\ 
		\hline
		b6 & -0.001043  &  0.0003085  & -0.0003057  \\
		\hline
		a7 & -0.001     & 0.0004384 & -0.0007851  \\
		\hline
		b7 &  0.0002438 & -0.0003108 &  0.0009376\\ 
		\hline 
		a8 &  0.0002026 & 0.0003586 & -0.0004529  \\
		\hline
		b8 & -0.0005033 & -0.0004005 & 0.0007278  \\
		\hline
		w  &  13.33     & 2.496 &  8.855\\
		\hline
	\end{tabular}
\end{table}
The linear regression analysis shows that the inhibition degree parameter, $\beta$, has a significant positive effect on the price of a contingent claim ($p < 0.01$), which means that as $\beta$ increases, the price of a contingent claim also increases. This result is consistent with our hypothesis that log-ergodicity enhances the value of a contingent claim by reducing the uncertainty and dependence of the price movements. The $R$-squared value of approximately 0.47 indicates that $\beta$ explains about 47$\%$ of the variation in the price of a contingent claim and suggests that log-ergodicity is a relatively good predictor of pricing contingent claims under ergodic market conditions.
\begin{figure}[H]
	\centering
	\begin{subfigure}{0.4\textwidth}
		\centering
		\hspace*{-0.6cm}
		\includegraphics[scale=0.45]{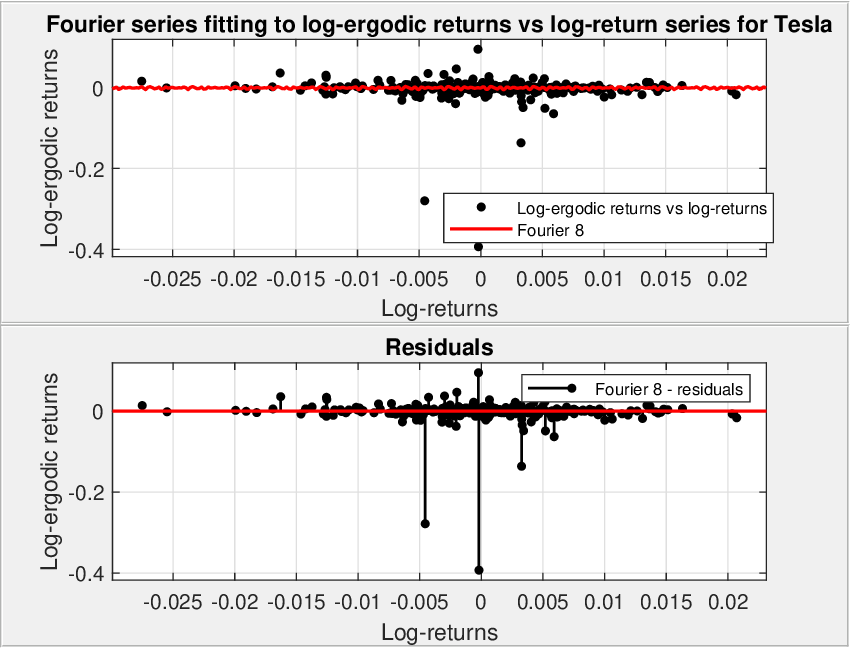}
		\label{F6a}
	\end{subfigure}
	\hfill
	\begin{subfigure}{0.4\textwidth}
		\centering
		\hspace*{-0.6cm}
		\includegraphics[scale=0.45]{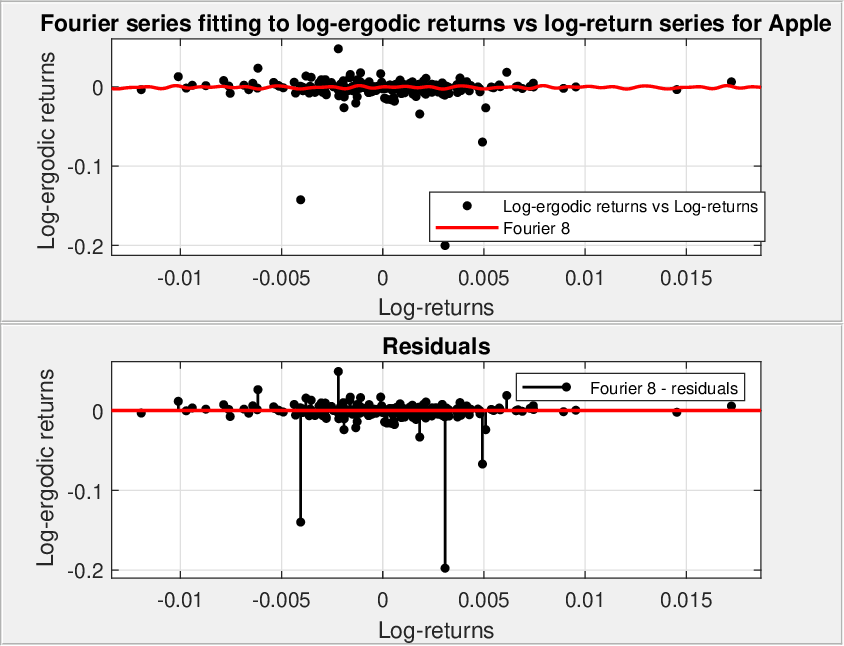}
		\label{F6b}
	\end{subfigure}
\hfill
\begin{subfigure}{0.4\textwidth}
	\centering
	\hspace*{-0.6cm}
	\includegraphics[scale=0.47]{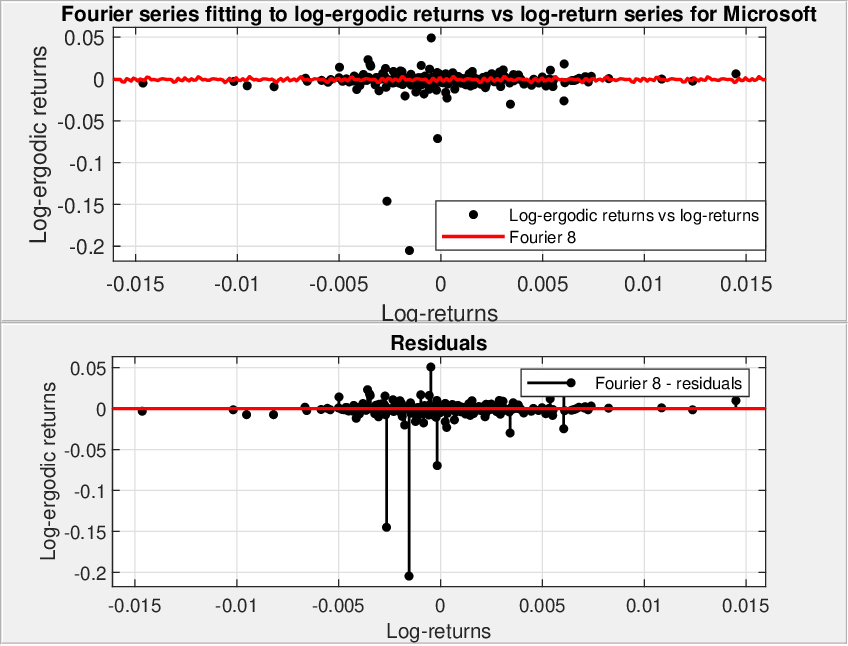}
	\label{F6c}
\end{subfigure}
	\caption{Scatter plots of log-ergodic returns vs. log-returns and the residuals for three stocks: TSLA, AAPL, and MSFT. The black dots represent the data points obtained from the empirical data analysis. The red curves represent the best fit of a Fourier series (with eight terms) to the data points.}
	\label{F6}
\end{figure}
\section{Conclusion and Future Research} \label{sec9}
In this paper, we have made the following contributions and findings: 

We introduced a new concept of log-ergodicity for positive stochastic processes, which is weaker than ergodicity but still captures some essential features of ergodic behavior in the mean. Also, we defined an ergodic maker operator that transforms a class of positive processes into a class of log-ergodic processes by scaling their deterministic and random components using a parameter that, in the case of price processes, reflects the degree of control exerted by market participants on these price processes.
Moreover, we showed that log-ergodic processes are usable for modeling financial markets with ergodic behavior in the mean, have applications in pricing contingent claims, and study market restrictions. Furthermore, we presented some empirical data analysis that supports our theoretical results using historical data from 2011 to 2021. We compared the performance and properties of log-ergodic models with geometric Brownian motion and stochastic volatility models. We used various statistical tests and measures to evaluate the usefulness of our work.

We also discussed some limitations and challenges of our approach and suggested some directions for future research. Some of them are: 

How does the concept of log-ergodicity affect other types of stochastic processes, such as fractional Brownian motion? How do other factors, such as market frictions ( transaction costs, taxes, dividends, and more), describe the ergodic behavior of financial markets when incorporating them into our models? How should we test the validity and robustness of log-ergodic models using more data sets from different markets and periods? How should we develop more efficient and accurate numerical methods for solving the partial differential equations derived from log-ergodic models?
\section*{Data Availability Statement}
All stock data used during this study are openly available from Yahoo Finance and Trading View websites as mentioned in the context.

\section*{Declaration of Interest}
The authors have no conflicts of interest to declare. Both authors have seen and agree with the contents of the manuscript, and there is no financial interest to report.
\bibliographystyle{plain}
\bibliography{sn-bibliography}

\begin{thebibliography}{10}

\bibitem{cycles}
Mark Aguiar and Gita Gopinath.
\newblock Emerging market business cycles: The cycle is the trend.
\newblock {\em Journal of political {E}conomy}, 115(1):69--102, 2007.

\bibitem{105}
Daniel Bernoulli.
\newblock Exposition of a new theory on the measurement of risk.
\newblock In {\em The {K}elly capital growth investment criterion: {T}heory and
  {P}ractice}, pages 11--24. World {S}cientific, 2011.

\bibitem{13}
Tomas Bj{\"o}rk.
\newblock {\em Arbitrage theory in continuous time}.
\newblock Oxford {U}niversity {P}ress, (4th ed.), 2020.

\bibitem{106}
{P}ierre Br{\'e}maud.
\newblock {\em Probability {T}heory and {S}tochastic {P}rocesses, {E}rgodic
  {P}rocesses}.
\newblock Springer, 2020.

\bibitem{76}
Ovidiu Calin.
\newblock An introduction to stochastic calculus with applications to finance.
\newblock {\em Ann {A}rbor}, 2012.

\bibitem{dc}
Fuqi Chen, Rogemar Mamon, and Matt Davison.
\newblock Inference for a mean-reverting stochastic process with multiple
  change points.
\newblock {\em Electrical {J}ournal of {S}tatistics}, 11:2199--2257, 2017.

\bibitem{cont}
Rama Cont and Peter Tankov.
\newblock {\em Financial modelling with jump processes}, volume~1.
\newblock {C}hapman \& {H}all/{CRC}, 2004.

\bibitem{as}
Hans F{\"o}llmer and Walter Schachermayer.
\newblock Asymptotic arbitrage and large deviations.
\newblock {\em Mathematics and {F}inancial {E}conomics}, 1:213--249, 2008.

\bibitem{83}
Liliana Gonzalez, John~G Powell, Jing Shi, and Antony Wilson.
\newblock Two centuries of bull and bear market cycles.
\newblock {\em International {R}eview of {E}conomics \& {F}inance},
  14(4):469--486, 2005.

\bibitem{35}
Robert~M Gray and Paul~C Shields.
\newblock Probability, {R}andom {P}rocesses, and {E}rgodic {P}roperties.
\newblock {\em S{IAM} {R}eview}, 36(1):143--144, 1994.

\bibitem{95}
Eric Hillebrand.
\newblock {\em Mean reversion models of financial markets}.
\newblock PhD thesis, Universit{\"a}t {B}remen, 2003.

\bibitem{bail}
Jennifer~R Horner.
\newblock Clogged systems and toxic assets: {N}ews metaphors, neoliberal
  ideology, and the {U}nited {S}tates ``{W}all {S}treet {B}ailout" of 2008.
\newblock {\em Journal of {L}anguage and {P}olitics}, 10(1):29--49, 2011.

\bibitem{87}
Ola H{\"o}ssjer and Arvid Sj{\"o}lander.
\newblock Sharp {L}ower and {U}pper {B}ounds for the {C}ovariance of {B}ounded
  {R}andom {V}ariables.
\newblock {\em ar{X}iv preprint ar{X}iv:2106.10037}, 2021.

\bibitem{114}
Olav Kallenberg.
\newblock {\em Foundations of modern probability}, volume~2.
\newblock Springer, 1997.

\bibitem{84}
P{\'e}ter Kevei.
\newblock Ergodic properties of generalized {O}rnstein--{U}hlenbeck processes.
\newblock {\em Stochastic {P}rocesses and their {A}pplications},
  128(1):156--181, 2018.

\bibitem{68}
Steve~G Kou.
\newblock Jump-diffusion models for asset pricing in financial engineering.
\newblock {\em Handbooks in operations research and management science},
  15:73--116, 2007.

\bibitem{kuni}
Hiroshi Kunita.
\newblock {\em Stochastic flows and stochastic differential equations},
  volume~24.
\newblock Cambridge {U}niversity {P}ress, 1997.

\bibitem{93}
Oesook Lee.
\newblock Exponential ergodicity and $\beta$-mixing property for generalized
  {O}rnstein-{U}hlenbeck processes.
\newblock {\em Theoretical {E}conomics {L}etters}, 2012.

\bibitem{104}
Robert~C Lowry.
\newblock A visible hand? {B}ond markets, political parties, balanced budget
  laws, and state government debt.
\newblock {\em Economics \& {P}olitics}, 13(1):49--72, 2001.

\bibitem{39}
Sean~P Meyn and Richard~L Tweedie.
\newblock Stability of {M}arkovian processes {I}: {C}riteria for discrete-time
  chains.
\newblock {\em Advances in {A}pplied {P}robability}, 24(3):542--574, 1992.

\bibitem{40}
Sean~P Meyn and Richard~L Tweedie.
\newblock Stability of {M}arkovian processes {II}: {C}ontinuous-time processes
  and sampled chains.
\newblock {\em Advances in {A}pplied {P}robability}, 25(3):487--517, 1993.

\bibitem{21}
Bernt {\O}ksendal.
\newblock {\em {S}tochastic {D}ifferential {E}quations. {A}n {I}ntroduction
  with {A}pplications}.
\newblock Springer, (6th ed.) 2003.

\bibitem{108}
Carlos~G Pacheco-Gonz{\'a}lez.
\newblock Ergodicity of a bounded {M}arkov chain with attractiveness towards
  the centre.
\newblock {\em Statistics \& probability letters}, 79(20):2177--2181, 2009.

\bibitem{90}
Ole Peters.
\newblock Optimal leverage from non-ergodicity.
\newblock {\em Quantitative {F}inance}, 11(11):1593--1602, 2011.

\bibitem{89}
Ole Peters.
\newblock The ergodicity problem in economics.
\newblock {\em Nature {P}hysics}, 15(12):1216--1221, 2019.

\bibitem{poll}
Mark Pollicott and Michiko Yuri.
\newblock {\em Dynamical systems and ergodic theory}.
\newblock Number~40. Cambridge {U}niversity {P}ress, 1998.

\bibitem{roy}
Halsey~Lawrence Royden and Patrick Fitzpatrick.
\newblock {\em Real analysis}, volume~2.
\newblock Macmillan {N}ew {Y}ork, 1968.

\bibitem{85}
Ren{\'e}~L Schilling.
\newblock An {I}ntroduction to {L}{\'e}vy and {F}eller {P}rocesses. {A}dvanced
  {C}ourses in {M}athematics-{CRM} {B}arcelona.
\newblock {\em arXiv e-prints}, pages ar{X}iv--1603, 2016.

\bibitem{59}
Cosma Shalizi.
\newblock Advanced {P}robability {II} or {A}lmost {N}one of the {T}heory of
  {S}tochastic {P}rocesses, 2007.

\bibitem{58}
Steven~E Shreve.
\newblock {\em Stochastic calculus for finance {II}: {C}ontinuous-time models},
  volume~11.
\newblock Springer, 2004.

\bibitem{57}
Karl Sigman.
\newblock {N}otes on {F}inancial {E}ngineering 4700: {N}otes on {B}rownian
  {M}otion, 2006.

\bibitem{103}
Nassim~Nicholas Taleb and Mark Spitznagel.
\newblock The {G}reat {B}ank {R}obbery.
\newblock {\em C{NN} {P}ublic {S}quare}, 2011.

\bibitem{62}
Chiraz Trabelsi.
\newblock {\em Ergodicity properties of affine term structure models and
  applications}.
\newblock PhD thesis, Universit{\"a}t Wuppertal, Fakult{\"a}t f{\"u}r
  Mathematik und Naturwissenschaften~…, 2018.

\bibitem{1}
Marcelo Viana and Krerley Oliveira.
\newblock {\em Foundations of ergodic theory}.
\newblock Cambridge {U}niversity {P}ress, 2016.

\bibitem{86}
Matthias Winkel.
\newblock Introduction to {L}évy processes.
\newblock {\em graduate lecture at the {D}epartment of {S}tatistics, {U}niv. of
  {O}xford}, 22, 2004.

\bibitem{100}
Zixuan Zhang, Michael Zargham, and Victor~M Preciado.
\newblock On modeling blockchain-enabled economic networks as stochastic
  dynamical systems.
\newblock {\em Applied {N}etwork {S}cience}, 5(1):1--24, 2020.

\end{thebibliography}


\end{document}